\documentclass[11pt,english]{article}

\pdfoutput=1

\usepackage[margin = 1.2in]{geometry}

\usepackage{amsthm}
\usepackage{amsmath}
\usepackage{amssymb}
\usepackage{setspace}
\usepackage{mathtools}
\usepackage{graphicx}
\usepackage[hidelinks]{hyperref}
\usepackage{cleveref}

\usepackage{lineno}
\usepackage{enumerate}
\usepackage{framed}
\usepackage{subcaption}

\usepackage{floatrow}

\theoremstyle{plain}

\newtheorem*{thm*}{Theorem}
\newtheorem{thm}{Theorem}
\Crefname{thm}{Theorem}{Theorems}

\newtheorem*{lem*}{Lemma}
\newtheorem{lem}[thm]{Lemma}
\Crefname{lem}{Lemma}{Lemmas}

\newtheorem*{claim*}{Claim}
\newtheorem{claim}[thm]{Claim}
\crefname{claim}{Claim}{Claims}
\Crefname{claim}{Claim}{Claims}

\newtheorem{prop}[thm]{Proposition}
\Crefname{prop}{Proposition}{Propositions}

\newtheorem{cor}[thm]{Corollary}
\Crefname{cor}{Corollary}{Corollaries}

\newtheorem{conj}[thm]{Conjecture}
\Crefname{conj}{Conjecture}{Conjectures}

\newtheorem{qn}[thm]{Question}
\Crefname{qn}{Question}{Questions}

\newtheorem{obs}[thm]{Observation}
\Crefname{obs}{Observation}{Observations}

\theoremstyle{definition}

\Crefname{prob}{Problem}{Problems}

\newtheorem{defn}[thm]{Definition}
\Crefname{defn}{Definition}{Definitions}

\theoremstyle{remark}

\renewenvironment{proof}[1][]{\begin{trivlist}
\item[\hspace{\labelsep}{\bf\noindent Proof#1.\/}] }{\qed\end{trivlist}}

\newcommand{\remove}[1]{}

\newcommand{\cfree}{$\{C_3, C_5\}$-free}

\newcommand{\suppose}{Let $G$ be a \cfree{} graph on $n$ vertices
with $\delta(G) > n/5$.}

\newcommand{\supposemax}{Let $G$ be a maximal \cfree{} graph on $n$ vertices
with $\delta(G) > n/5$.}

\newcommand{\wellbehaved}{well-behaved}

\newcommand{\MH}{\mathcal{H}}
\newcommand{\eps}{\varepsilon}

\DeclareMathOperator{\Forb}{Forb}
\DeclareMathOperator{\modulo}{mod}

\begin{document}


\title{The homomorphism threshold of \cfree{} graphs}
\date{\vspace{-5ex}}
\author{
    Shoham Letzter\thanks{
        Department of Pure Mathematics and Mathematical Statistics, 	
        University of Cambridge, 
        Wilberforce Road, 
        CB3\;0WB Cambridge, 
        UK;
        e-mail:
        \mbox{\texttt{s.letzter@dpmms.cam.ac.uk}}\,.
    } 
    \and
    Richard Snyder\thanks{
    	Department of Mathematical Sciences, University 
    	of Memphis, Memphis, Tennessee;
    	e-mail: \mbox{\texttt{rsnyder1@memphis.edu}}\,
     }
}

\maketitle
\singlespace

\begin{abstract}

    \setlength{\parskip}{\medskipamount}
    \setlength{\parindent}{0pt}
    \noindent
	We determine the structure of $\{C_3, C_5\}$-free graphs with $n$ vertices
	and minimum degree larger than $n/5$: such graphs are homomorphic to the
	graph obtained from a $(5k - 3)$-cycle by adding all chords of length
	$1(\modulo 5)$, for some $k$. This answers a question of Messuti and
	Schacht.  We deduce that the homomorphism threshold of $\{C_3, C_5\}$-free
	graphs is $1/5$, thus answering a question of Oberkampf and Schacht.

\end{abstract}
\onehalfspace

\section{Introduction} \label{sec:intro}

	We are interested in the structure of graphs of high minimum degree which
	forbid specific subgraphs. For a fixed graph $H$, a graph is said to be
	$H$-\emph{free} if it does not contain $H$ as a subgraph.  Let $\Forb(H)$
	denote the class of $H$-free graphs, and let $\Forb_{n}(H)$ denote the
	class of $n$-vertex graphs in $\Forb(H)$. Furthermore, let $\Forb(H,d)$
	denote the class of $H$-free graphs $G$ with minimum degree at least
	$d|V(G)|$.  Analogous definitions hold if we replace $H$ by some family
	$\mathcal{H}$ of graphs. Finally, we say that a graph $G$ is homomorphic to
	a graph $H$ if there exists a map $f: V(G) \rightarrow V(H)$ such that
	$f(u)f(v)\in E(H)$ whenever $uv \in E(G)$.  For example, $G$ is homomorphic
	to $K_r$ if and only if $\chi(G) \le r$.
    
	A classical result of Andr{\'a}sfai, Erd{\H{o}}s and S{\'o}s
	\cite{andrasfai-erdos-sos} states that if $G$ is a $K_{r+1}$-free graph on
	$n$ vertices with minimum degree $\delta(G) > \frac{3r-4}{3r-1}n$, then $G$
	is $r$-colourable. This result can be viewed as a significant strengthening
	of the following fact, which is a consequence of Tur\'an's theorem: the
	minimum degree of a $K_{r + 1}$-free graph on $n$ vertices is at most $(1 -
	1/r)n$.  Note also here that the chromatic number $\chi(G)$ of $G$ is
	bounded by a constant independent of $n$.  In general, one may ask whether
	or not this behaviour persists when the minimum degree condition is
	weakened. Along these lines, H{\"a}ggkvist \cite{haggkvist} showed that any
	$n$-vertex triangle-free graph of minimum degree greater than $3n/8$ is
	homomorphic to a $5$-cycle, and accordingly has chromatic number at most
	$3$. Note that this is indeed an extension of the
	Andr{\'a}sfai-Erd{\H{o}}s-S{\'o}s result when the minimum degree condition
	is weakened, since a balanced blow-up of a $5$-cycle exhibits the tightness
	of that result.  Jin \cite{jin1} took up the investigation and
	significantly extended the work of H\"aggkvist: he proved that for all
	$1\leq k \leq 9$, any $n$-vertex triangle-free graph with minimum degree
	larger than $\frac{k+1}{3k+2}n$ is homomorphic to the graph $F_{k}^2$,
	which is obtained by adding all chords of length $1(\modulo3)$ to a cycle
	of length $3k - 1$. Observe that $F_k^2$ is triangle-free and
	$3$-colourable for every $k$. The graphs $F_k^2$ are a special case of a
	larger family of graphs, $F_k^\ell$, which we shall discuss shortly.  We
	note that Jin's result \cite{jin1} is best possible, in the sense that such
	a statement does not hold for $k = 10$. Indeed, by taking a suitably chosen
	unbalanced blow-up of the Gr\"otzsch graph (also known as the
	Mycielski graph, see e.g.\ \cite{chen-jin-koh}) one can obtain a
	triangle-free graph on $n$ vertices and minimum degree $\lfloor 10n/29
	\rfloor$ which is not $3$-colourable, so in particular it is not
	homomorphic to $F_k^2$ for any $k$.  Building on this work, Chen, Jin,
	and Koh \cite{chen-jin-koh} showed, in particular, that any $n$-vertex
	$3$-colourable triangle-free graph $G$ with $\delta(G) > n/3$ is
	homomorphic to $F_k^2$, for some $k$. Again, the Gr\"otzsch graph shows
	that the assumption that the graph is $3$-colourable is necessary. 
    
	In general, one may ask for the \emph{smallest} minimum degree condition we
	may impose on an $H$-free graph which guarantees that it has bounded
	chromatic number. To be precise, this prompts us to define the
	\emph{chromatic threshold} $\delta_{\chi}(H)$ of a graph $H$:
    \[ 
		\delta_{\chi}(H) = 
        \inf \{d: \text{there exists $C=C(H,d)$ such that if $G \in \Forb(H,d)$,
		then $\chi(G) \leq C$}\}. 
	\]
	In other words, $\delta_{\chi}(H)$ is the infimum over all $d \in [0, 1]$
	such that every $H$-free graph on $n$ vertices and with minimum degree at
	least $dn$ has bounded chromatic number (independent of $n$).  This
	definition was implicit in the works of Andr{\'a}sfai \cite{andrasfai} and
	Erd{\H{o}}s and Simonovits \cite{erdos-simonovits}, and was first
	explicitly formulated by \L{}uczak and Thomass{\'e} \cite{luczak-thomasse}.
    
	For every $\varepsilon > 0$, Hajnal (appearing in \cite{erdos-simonovits})
	constructed graphs in $\Forb(K_3, 1/3-\varepsilon)$ with arbitrarily large
	chromatic number, thereby proving the bound $\delta_{\chi}(K_3)\geq 1/3$.
	Thomassen \cite{thomassen} thereafter established the matching upper bound,
	showing that $\delta_{\chi}(K_3) = 1/3$.  In fact, Brandt and Thomass{\'e}
	\cite{brandt-thomasse} strengthened this by showing that triangle-free
	graphs of minimum degree larger than $n/3$ have chromatic number at most
	four, answering a question of Erd{\H{o}}s and Simonovits
	\cite{erdos-simonovits}. Extensions of these results were obtained by
	several authors \cite{goddard-lyle, nikiforov}, who showed that
	$\delta_{\chi}(K_r) = \frac{2r-5}{2r-3}$.  Finally, building off of ideas
	of {\L}uczak and Thomass{\'e} \cite{luczak-thomasse} and Lyle \cite{lyle},
	Allen, B\"ottcher, Griffiths, Kohayakawa and Morris \cite{allen-et-al}
	determined the value of $\delta_{\chi}(H)$ for every graph $H$ with
	$\chi(H) >2$.
    
	Note that the results of H\"aggkvist \cite{haggkvist}, Jin \cite{jin1}, and
	Chen, Jin, and Koh \cite{chen-jin-koh} mentioned earlier not only show that
	triangle-free graphs of large enough minimum degree have bounded chromatic
	number, but that they are actually homomorphic to some specific
	$3$-colourable triangle-free graph.  One may ask then, with respect to the
	above discussion, whether we can replace the property of having bounded
	chromatic number with the property of admitting a homomorphism to a graph
	of bounded order with additional properties. This question was posed by
	Thomassen \cite{thomassen} in the specific case of triangle-free graphs,
	and motivated Oberkampf and Schacht \cite{oberkampf-schacht} to introduce
	the \emph{homomorphism threshold} $\delta_{\hom}(H)$ of a graph $H$:
    \begin{align*}
		\delta_{\hom}(H) = \inf\{d: \exists\, C=C(H,d) \text{ s.t. } &\forall\, G \in \Forb(H,d) \\
    	&\exists\, G' \in \Forb_C(H) \text{ s.t.~} G \text{ is homomorphic to } G'\}. 
    \end{align*}
	In words, $\delta_{\hom}(H)$ is the infimum over all $d \in [0,1]$
	such that every $H$-free graph with $n$ vertices and minimum degree at
	least $dn$ is homomorphic to an $H$-free graph of bounded order
	(independent of $n$).  Note that the definition of $\delta_{\hom}(H)$
	extends naturally if we replace $H$ by a family $\mathcal{H}$ of graphs.
    
	\L{}uczak \cite{luczak} proved that $\delta_{\hom}(K_3) \leq 1/3$. Note
	that if $G$ is homomorphic to $G'$, then $\chi(G) \le |V(G')|$.
	Accordingly, we always have $\delta_{\hom}(H) \ge \delta_{\chi}(H)$, and
	so, since $\delta_{\chi}(K_3) = 1/3$, it follows that $\delta_{\hom}(K_3) =
	1/3$. This result was extended by Goddard and Lyle \cite{goddard-lyle} to
	$K_r$-free graphs for $r\geq 4$, and, in particular, we know that
	$\delta_{\hom}(K_r) = \delta_{\chi}(K_r) = \frac{2r-5}{2r-3}$.  Oberkampf
	and Schacht \cite{oberkampf-schacht} gave a new proof of this result
	avoiding the Regularity Lemma (which was used in \L{}uczak's proof), and
	asked for the determination of the homomorphism threshold of the odd cycle,
	$\delta_{\hom}(C_{2\ell-1})$, and $\delta_{\hom}(\{C_3,\ldots ,
	C_{2\ell-1}\})$ for $\ell \ge 3$.  As our first main result, we
	determine the value of the second of these two parameters in the case
	$\ell = 3$.
    \begin{thm} \label{thm:hom-threshold}
        The homomorphism threshold of $\{C_3, C_5\}$ is $1/5$.
    \end{thm}
	In other words, \Cref{thm:hom-threshold} states that, for every $\eps > 0$,
	if $G$ is a \cfree{} graph on $n$ vertices and minimum degree at least
	$(1/5 + \eps)n$, then $G$ is homomorphic to a \cfree{} graph of order at
	most $C$, where $C$ depends on $\eps$ but not on $n$.  We also establish an
	upper bound on $\delta_{\hom}(C_5)$.  This is a consequence of
	\Cref{thm:hom-threshold}, since $C_5$-free graphs of large minimum degree
	end up being triangle-free as well. In particular, we have the following.
    
    \begin{cor}\label{cor:hom-threshold-C-5}
        The homomorphism threshold of $C_5$ is at most $1/5$. 
    \end{cor}
    
	In fact, we are able to say much more about the structure of \cfree{}
	graphs with $n$ vertices and minimum degree larger than $n/5$.  First we
	need to define a family of graphs, sometimes known as \emph{generalised
	Andr{\'a}sfai graphs}.  For integers $k \ge 1$ and $\ell \ge 2$, denote
	by $F_k^\ell$ the graph obtained from a $((2\ell-1)(k-1)+2)$-cycle (an
	edge, when $k = 1$) by adding all chords joining vertices at distances
	$j(2\ell-1)+1$ for $j=0, 1, \ldots , k-1$. Alternatively,
	$F_k^{\ell}$ can be defined as the complement of the $(\ell -
	1)$-th power of a cycle of length $(2\ell-1)(k-1)+2$. For $\ell =
	2$, these graphs were considered by Erd\H{o}s \cite{erdos} and Andr\'asfai
	\cite{andrasfai,andrasfai-2}. It is not difficult to check that
	$F_k^\ell$ is $k$-regular, maximal $\{C_3, \ldots , C_{2\ell - 1}\}$-free,
	and $3$-colourable.  For our purposes, $\ell$ will always be $3$ and we
	shall write $F_k$ instead of $F_k^3$ for simplicity. In particular, $F_1$ is
	an edge, $F_2$ is a $C_7$ (a cycle of length $7$) and $F_3$ is the graph
	obtained by adding all diagonals to a $C_{12}$ (by a \emph{diagonal} in an
	even cycle $C_{2\ell}$, $\ell \geq 2$, we mean an edge joining vertices at
	distance $\ell$ along the cycle). This graph is also known as the M{\"o}bius
	ladder on $12$ vertices (see \Cref{fig:mobius}).
 
    \begin{figure}[h]
        \ffigbox
       	{
        	\begin{subfloatrow}
                \ffigbox[.35\textwidth]
                {
                    \includegraphics[scale = .8]{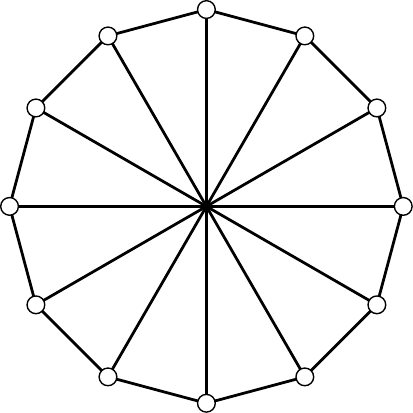}
                }
                {
                    \caption{the M{\"o}bius ladder $F_3$}
                    \label{fig:mobius} 
                }
                \ffigbox[.35\textwidth]
                {
                    \includegraphics [scale = .8]{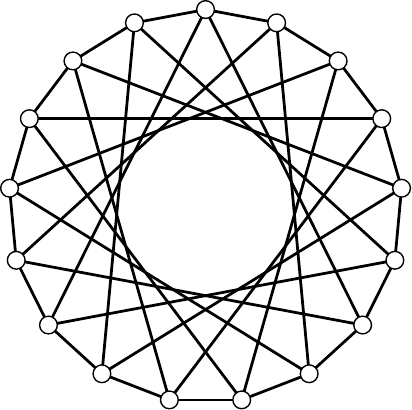}
                }
                {
                    \caption{$F_4$}
                    \label{fig:F4}
                }
            \end{subfloatrow}
        }
        {
            \caption{examples of graphs $F_k$}
            \label{fig:F-k-examples}
       	}
    \end{figure}    

	As our second main result, we determine the structure of \cfree{} graphs on
	$n$ vertices with minimum degree larger than $n/5$, thus answering a
	question of Messuti and Schacht \cite{messuti-schacht}.   
    \begin{thm} \label{thm:hom-to-family}   
        \suppose{}
        Then $G$ is homomorphic to $F_k$, for some $k$.
    \end{thm}
    
	We remark that the analogue of this result for graphs of higher odd-girth
	does not hold in general. We discuss this further in the final section of
	this paper.
    
    \subsection{Organisation and Notation}
    
		The remainder of this paper is organised as follows. In
		\Cref{sec:overview} we shall provide an outline of the technical
		results needed to prove our main theorem. Many of these state that
		certain subgraphs cannot appear in maximal \cfree{} graphs of minimum
		degree larger than $n/5$. In the next three sections
		(\Cref{sec:induced-6-cycles} to \Cref{sec:7-cycles-3-path}) we shall
		prove each of these technical results. In \Cref{sec:main-theorem}, we
		deduce our main theorem, \Cref{thm:hom-to-family}. Finally,
		\Cref{sec:hom-thresholds} includes our results concerning homomorphism
		thresholds, \Cref{thm:hom-threshold} and \Cref{cor:hom-threshold-C-5}. 
		
		Our notation is standard. In particular, for a graph $G$, we use
		$|V(G)|$ to denote the number of vertices of $G$, $V(G)$ denotes the
		vertex set, $E(G)$ the edge set, and $\delta(G)$ denotes the minimum
		degree. For a vertex $v$, $N_G(v)$ denotes the neighbourhood of $v$,
		and for a subset $X\subseteq V(G)$, $N_G(v,X)$ denotes the
		neighbourhood of $v$ in $X$, i.e.~$N_G(v, X) = N_G(v) \cap X$. We shall
		often omit the use of the subscript `$G$'. If $X, Y \subseteq V(G)$,
		then we say an edge $e$ is an $X-Y$ \emph{edge} if one endpoint of $e$
		is in $X$, the other in $Y$.  If $X$=$\{x\}$, then we simply say $e$ is
		an $x-Y$ edge.  We denote by $(v_1\ldots v_\ell)$ the cycle on vertices
		$v_1, \ldots v_\ell$ taken in this order.  Similarly, we denote by
		$v_0\ldots v_\ell$ the path on vertices $v_0, \ldots , v_\ell$ taken in
		this order. A cycle (path) with $\ell$ edges is an $\ell$\emph{-cycle}
		($\ell$\emph{-path}).
     
\section{Overview} \label{sec:overview}
    
	In this section we provide a tour through the technical results needed to
	establish our main theorem. Note that in proving \Cref{thm:hom-to-family}
	we may assume our graph is maximal \cfree{}.  Accordingly, the following
	results concern maximal \cfree{} graphs.  The main tool needed for the
	proof of \Cref{thm:hom-to-family} is the following result.
    
    \begin{thm} \label{thm:vx-cycle}   
        \supposemax{}
        Then every vertex in $G$ has a neighbour in every $7$-cycle in $G$. 
    \end{thm}    
    
	We remark that Jin \cite{jin} proved the analogous theorem for $5$-cycles
	in triangle-free graphs of large enough minimum degree.  In order to
	establish \Cref{thm:vx-cycle} we shall need a sequence of lemmas which show
	that certain subgraphs cannot appear in maximal \cfree{} graphs of large
	minimum degree. The first of these lemmas, which shows that \cfree{} graphs
	with large minimum degree do not have induced $6$-cycles, proves very
	useful, and we shall use it throughout the paper.  Brandt and Ribe-Baumann
	\cite{brandt-ribe} mention it without proof.
    \begin{lem} \label{lem:induced-cycle}
        \supposemax{}
        Then $G$ does not contain an induced $6$-cycle.
    \end{lem}

    We shall also need the fact that a `partial' M{\"o}bius ladder cannot 
    appear as a subgraph in the graphs we consider. More precisely, 
    we need the following lemma.
    
    \begin{lem} \label{lem:cycle-chords}
        \supposemax{}
        If $(x_1\ldots x_{12})$ is a $12$-cycle with two consecutive diagonals $x_1x_7$ and $x_2x_8$ present. Then either 
        $(x_1\ldots x_{12})$ or $(x_2x_3\ldots x_7x_1x_{12}\ldots x_8)$ induces a M{\"o}bius ladder.
    \end{lem}
    
    We note, and prove, the following useful corollary of \Cref{lem:cycle-chords}.
    
    \begin{cor} \label{cor:vx-cycle-two-neighs}
        \supposemax{}
        If $u$ is a vertex with no neighbours in a $7$-cycle $C$, then $u$
        has no neighbour with two neighbours in $C$.
    \end{cor}

    \begin{proof}
		Suppose that $C = (x_1 \ldots x_7)$ and $u$ has no neighbours in $C$,
		but a neighbour $v$ of $u$ has two neighbours in $C$. Say, $v$ is
		adjacent to $x_2$ and $x_7$ (see \Cref{prop:well-behaved} below).
        
        \begin{figure}[h]
			\centering
			\includegraphics[scale = .85]{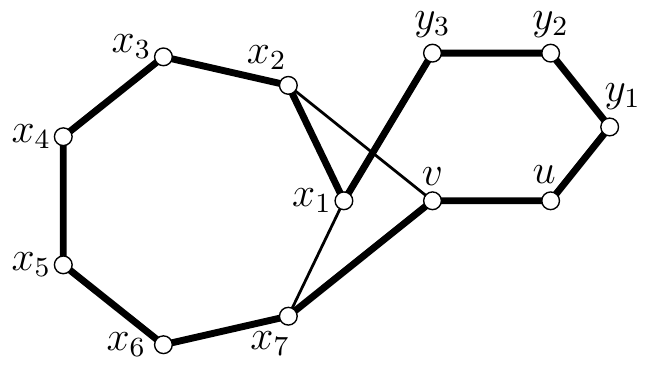}
			\caption{}
			\label{fig:seven-cycle-u-v-cycle}
        \end{figure}            
         
		Since $u$ is not adjacent to $x_1$ and $G$ is maximal $\{C_3,
		C_5\}$-free, there must be a path of length $2$ or $4$ between them;
		but a path of length $2$ is impossible (it will complete the path $u v
		x_2 x_1$ to a $5$-cycle), so there is a $4$-path $u y_1 y_2 y_3 x_1$.
		One may check that none of $y_1, y_2, y_3$ is equal to one of the
		vertices of $C$ or to $v$ (see \Cref{fig:seven-cycle-u-v-cycle}).  But
		then $(x_1 \ldots x_7 v u y_1 y_2 y_3)$ is a $12$-cycle with two
		consecutive diagonals $x_1 x_7$ and $x_2 v$.  It follows from
		\Cref{lem:cycle-chords} that all diagonals in the cycle must be present
		(or we need to consider the $12$-cycle $(x_2 \ldots x_7 x_1 y_3 y_2 y_1
		u v)$ with diagonals $x_1 x_2$ and $x_7 v$). In particular, $u$ has a
		neighbour in $C$, a contradiction.
    \end{proof}
    
	Finally, in order to prove \Cref{thm:vx-cycle}, we establish the following
	lemma, which is the last of our results regarding forbidden subgraphs
	in maximal \cfree{} graphs of large minimum degree.
    
    \begin{lem} \label{lem:two-cycles}
        \supposemax{}
        Then $G$ does not contain, as an induced graph, the graph obtained by
        two $7$-cycles whose intersection is a path of length $3$ (see \Cref{fig:two-cycles}).
    \end{lem}

	Before proceeding to the proofs of the above forbidden subgraph lemmas, we
	shall show how to prove \Cref{thm:vx-cycle} using
	\Cref{lem:induced-cycle,lem:cycle-chords,lem:two-cycles}. The proofs
	of these lemmas shall be deferred to
	Sections~\ref{sec:induced-6-cycles},~\ref{sec:12-cycles-few-chords}
	and~\ref{sec:7-cycles-3-path}, respectively. In order to aid in their
	proofs, we introduce the following definition.
    
    \begin{defn} \label{def:well-behaved}
		A subgraph $H$ of a graph $G$ is called \emph{\wellbehaved{}} (in $G$)
		if for every vertex $u$ in $G$, there is a vertex $v$ in $H$, such that
		$N_G(u, H) \subseteq N_H(v)$.
    \end{defn}

	In particular, this implies that $G[H \cup \{u\}]$ is homomorphic to $H$
	for every $u \in V(G)$.  Many of the subgraphs we consider are actually
	\wellbehaved{} (in their respective host graphs). For example, we note the
	following useful proposition.
    
    \begin{prop}\label{prop:well-behaved}
		Let $\ell \ge 2$ be an integer and let $G$ be a $\{C_3, \ldots,
		C_{2\ell - 1}\}$-free graph. Then $C_{2\ell+1}$ is well-behaved in $G$.
	\end{prop}
	\begin{proof}
		Let $\ell \ge 2$ and let $C = (x_1\ldots x_{2\ell+1})$ be a
		$(2\ell+1)$-cycle in $G$. Suppose without loss of generality that $w
		\in V(G) \setminus V(C)$ is joined to $x_1$.  We claim that either
		$N(w, C) \subseteq N_{C}(x_2)$ or $N(w, C) \subseteq
		N_{C}(x_{2\ell+1})$.  Let $w'$ be another neighbour of $w$ in $C$ and
		suppose to the contrary that $w' \neq x_3, x_{2\ell}$. Let $P$ denote
		the path $x_1x_2x_3\ldots w'$ and $P'$ denote the path
		$x_1x_{2\ell+1}x_{2\ell}\ldots w'$. Now, note that $l(P) \le (2\ell+1)
		- 3 = 2\ell - 2$ and similarly $l(P') \le 2\ell-2$ (here $l(P)$
		denotes the length of $P$). Moreover, one of $P, P'$ must have
		odd length, say $P$. But then the cycle $(wx_1Pw')$ is odd and has
		length at most $2\ell -1$, a contradiction.
	\end{proof}    

	We need the following observation before proving \Cref{thm:vx-cycle}.
		
    \begin{obs} \label{obs:one-common-neigh}
		\supposemax{} Suppose that $u$ has no neighbour in a $7$-cycle $C$.
		Then $u$ has a common neighbour
		with at most one of the vertices in $C$.
    \end{obs}
	\begin{proof}
		
		Suppose that $u$ has no neighbour in the cycle $C = (x_1 \ldots x_7)$.
		Furthermore, suppose that $u$ has a common neighbour $v$ with $x_1$.
		By symmetry, it suffices to show that $u$ has no common neighbour with
		$x_2$, $x_3$ or $x_4$.  It easily follows that $u$ and $x_2$ have no
		common neighbour (otherwise, a cycle of length $3$ or $5$ is formed).
		Suppose that $u$ and $x_3$ have a common neighbour $w$. Observe that $w
		\neq v$ by \Cref{cor:vx-cycle-two-neighs}.  Consider the $6$-cycle $(v
		u w x_3 x_2 x_1)$.  Recall that $G$ has no induced $6$-cycles; thus one
		of $u x_2, v x_3, w x_1$ is an edge in $G$.  But $u x_2$ is not an
		edge, by the assumption that $u$ has no neighbour in $C$, and if one of
		$v x_3$ and $w x_1$ is an edge, a contradiction to
		\Cref{cor:vx-cycle-two-neighs} is reached.  Finally, if $u$ and $x_4$
		have a common neighbour $w$ (which, as before, is not equal to $v$),
		then the set $\{u, v, w, x_1, \ldots, x_7\}$ induces a graph that
		consists of two $7$-cycles whose intersection is a path of length $3$,
		contradicting \Cref{lem:two-cycles}.
	\end{proof}
    
    \begin{proof}[ of \Cref{thm:vx-cycle}]
                
		Suppose that the theorem is false and choose a vertex $u$ and a
		$7$-cycle $C$ which minimize the distance between $u$ and $C$ such that
		$u$ has no neighbour in $C$. Since $G$ must be connected, it easily
		follows that there is a path of length two between $u$ and $C$.
		Therefore, we may assume without loss of generality that $u$ has no
		neighbour in the $7$-cycle $C = (x_1 \ldots x_7)$ and $v$ is a common
		neighbour of $u$ and $x_1$.  Since $u$ is not joined to $x_2$ and $G$
		is maximal \cfree, there is a $4$-path $u y_1 y_2 y_3 x_2$ between $u$
		and $x_2$ (a $2$-path would create a $5$-cycle).  We note that $y_1$
		cannot be joined to $x_1$, otherwise a $5$-cycle is formed, so in
		particular $y_1 \neq v$. Thus, by \Cref{obs:one-common-neigh}, $y_1$
		has no neighbours in $C$.  We note that no two of the four vertices
		$\{u, x_2, x_3, x_6\}$ have a common neighbour; this follows from
		\Cref{obs:one-common-neigh} and the assumption that $G$ is \cfree.  It
		follows from the minimum degree condition that $y_1$ has a common
		neighbour with one of $u, x_2, x_3, x_6$.  But $y_1$ does not have a
		common neighbour with either $u$ or $x_2$ (otherwise, a cycle of length
		$3$ or $5$ if formed).  Thus $y_1$ has a common neighbour with either
		$x_3$ or $x_6$. Assume that $y_1$ has a common neighbour with $x_3$
		(with $x_6$). Then, by \Cref{obs:one-common-neigh}, $y_1$ has no common
		neighbours with any other vertex in $C$.  It follows that no two of the
		vertices in $\{u, y_1, x_2, x_5, x_6\}$ (in $\{u, y_1, x_3, x_4,
		x_7\}$) have a common neighbour, a contradiction to the minimum degree
		condition.
    \end{proof}

	In the next three sections we shall prove
	\Cref{lem:induced-cycle,lem:cycle-chords,lem:two-cycles}. The general
	strategy is the following. We want to show that some graph $F$ cannot
	appear in a maximal \cfree{} graph $G$ of large minimum degree.  If $F$ is
	a subgraph of $G$, and if every vertex has a `small' number of neighbours
	in $F$, then double counting the edges between $V(F)$ and $V(G)\setminus
	V(F)$ will produce a contradiction with the minimum degree condition. Often
	the original target graph $F$ does not satisfy this goal, and we shall need
	to pass to some suitable subgraph of $F$ which meets our needs. This
	requires detailed analysis of the possible neighbourhoods of vertices in
	$F$ (or some subgraph of $F$).    

\section{No induced $6$-cycles} \label{sec:induced-6-cycles}
	
   Brandt and Ribe-Baumann~\cite{brandt-ribe} stated that maximal \cfree{}
   graphs of high minimum degree forbid induced $6$-cycles. However, they did
   not provide a proof and for completeness we provide one in this section. 
   
   \begin{proof}[ of \Cref{lem:induced-cycle}] 
		Suppose $G$ contains an induced $6$-cycle $C= (x_1\ldots x_6)$. By the
		edge-maximality of $G$ there are three $4$-paths $P_{14} =
		x_1y_1y_2y_3x_4$, $P_{25} = x_2z_1z_2z_3x_5$, and $P_{36} =
		x_3w_1w_2w_3x_6$. It is easily verified that all the vertices in the
		union of these paths are distinct. Denote by $H$ the graph induced on
		$V(C) \cup V(P_{14})\cup V(P_{25})\cup V(P_{36})$ (see
		\Cref{fig:forbid-ind-6-cycle}).  We shall show that $G$ cannot
		contain $H$ as a subgraph. The proof breaks into two cases:
    
		\begin{enumerate}
			\item \label{itm:ind-6-cycle-middle}
				At least two vertices from $\{y_2, z_2, w_2\}$ have a common
				neighbour.
			\item \label{itm:ind-6-cycle-no-middle}
				No pairs from $\{y_2, z_2, w_2\}$ have a common neighbour.	
		\end{enumerate}
		
		\begin{figure}[h]
			\centering
			\includegraphics[scale = .85]{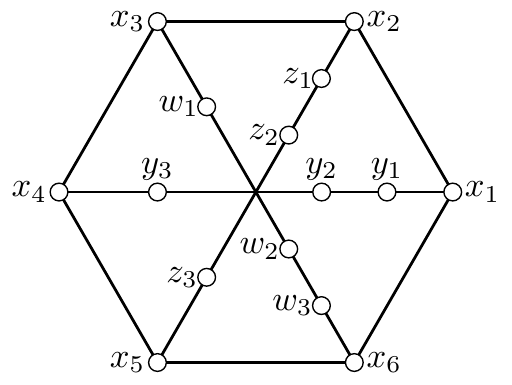}
			\caption{Constructing $H$ from an induced $6$-cycle}
			\label{fig:forbid-ind-6-cycle}
		\end{figure}
					
		In each case, we shall find a $10$-vertex subgraph of $H$ for which
		every vertex of $G$ has at most two neighbours in $H$. We reach a
		contradiction to the minimum degree condition on $G$ via double
		counting the edges between this subgraph and the rest of $G$.

		\subsection{Case \ref{itm:ind-6-cycle-middle}}
			
			Suppose, without loss of generality, that $y_2$ and $z_2$ have a
			common neighbour $v$.  Denote by $H'$ the $10$-vertex graph induced
			on $V(H)\setminus V(P_{36})$.
			
			\begin{claim}\label{claim:ind-6-cycle-claim1}
				Every vertex of $G$ has at most two neighbours in $H'$.
			\end{claim}
		
			\begin{proof}
				Observe that $x_1$ and $x_4$ cannot have a common neighbour,
				else a $5$-cycle is formed. Thus, a vertex in $G$ can have at
				most two neighbours in $C\setminus \{x_3, x_6\}$, and if it has
				two such neighbours, then it is joined either to both $x_2$ and
				$x_4$ or to both $x_1$ and $x_5$. It is then routine to check
				that such a vertex cannot be joined to any other vertex
				of $H'$:  all cases lead to a triangle or pentagon in $G$.
				
				Let us now consider vertices which have precisely one neighbour
				in $C\setminus \{x_3, x_6\}$. By symmetry, let $u$ be a vertex
				joined to $x_1$.  We claim that $u$ can be joined to at most
				one other vertex in $H'$, which must be from $\{y_2,
				z_1\}$.  Indeed, it is easy to verify that $u$ cannot be joined
				to any vertices of $H'\setminus \{y_2, z_1\}$, since these
				cases lead to a triangle or pentagon in $G$.  Suppose $u$ is
				adjacent to both $y_2$ and $z_1$. This, however, produces the
				$5$-cycle $(y_2uz_1z_2v)$, a contradiction.
				
				Finally, we consider vertices which have no neighbour in
				$C\setminus \{x_3, x_6\}$.  First, note that if a vertex $u$ is
				joined to $y_2$, then its only other possible neighbour in $H'$
				is $z_2$ (and the same claim holds with the roles of $y_2$ and
				$z_2$ reversed). For example, if $u$ is joined to $y_2$ and
				$z_3$, the $5$-cycle $(y_2uz_3z_2v)$ is produced. One may
				dispense with the other cases similarly.  On the other hand,
				both pairs $\{y_3, z_3\}$ and $\{y_1, z_1\}$ do not have
				any common neighbours. It follows that any vertex with no
				neighbour in $C \setminus \{x_3, x_6\}$ has at most two
				neighbours in $H'$, and this finishes the proof of
				\Cref{claim:ind-6-cycle-claim1}.				
			\end{proof}

			Let us bound the number of edges between $V(H')$ and $V(G)\setminus
			V(H')$ in two ways. Using the minimum degree condition and
			\Cref{claim:ind-6-cycle-claim1}, we have that every vertex in $H'$
			has at most two neighbours in $H'$, and therefore has more than
			$n/5 - 2$ neighbours outside of $H'$. Thus, there are more than
			$10(n/5 - 2) = 2(n - 10)$ such edges. On the other hand,
			\Cref{claim:ind-6-cycle-claim1} implies that there are at most
			$2(n-10)$ such edges, a contradiction. This completes the proof of
			\Cref{lem:induced-cycle} under Case \ref{itm:ind-6-cycle-middle}.
            
	   \subsection{Case \ref{itm:ind-6-cycle-no-middle}}
	   
			Suppose the condition in Case~\ref{itm:ind-6-cycle-no-middle}
			holds; that is, no pairs from $\{y_2, z_2, w_2\}$ have a common
			neighbour. We begin by examining the size and structure of possible
			neighbourhoods in $H$ of vertices of $G$.
			
			Let $u$ be a vertex which is not joined to any vertex in $C$. If
			$u$ is joined to a middle vertex, say $y_2$, then by assumption it
			cannot be adjacent to $z_2$ or $w_2$. Further, $u$ cannot be joined
			to $y_1$ or $y_3$ (else, a triangle is formed), has at most one
			neighbour in $\{z_3, w_3\}$, and has at most one neighbour in
			$\{z_1, w_1\}$. Thus, $u$ has at most three neighbours in
			$H$. Similarly, one may verify that if $u$ has no neighbour in
			$\{y_2, z_2, w_2\}$, then $u$ has at most three neighbours in
			$H$ as well.
			
			Suppose $u$ is a vertex joined to two vertices of $C$. Say,
			by symmetry, that $u$ is joined to $x_2$ and $x_4$. Then it is easy
			to check that the only other possible neighbour of $u$ in $H$ is
			$w_1$.  Hence $u$ has at most three neighbours in $H$.
			
			If $u$ is a vertex joined to three neighbours of $C$, then
			(up to relabelling) $u$ is adjacent to all vertices in  $\{x_1,
			x_3, x_5\}$, and one may verify that $u$ can have no further
			neighbour in $H$. Thus, such vertices have at most three
			neighbours in $H$.
			
			Only one case remains: suppose $u$ has precisely one neighbour in
			$C$, and, by symmetry, suppose this neighbour is $x_1$. In this
			case, $u$ may be joined to all vertices in $\{y_2, z_1, w_3\}$.
			Accordingly, $u$ has at most four neighbours in $H$.
			
			Now, if every vertex of $G$ has at most three neighbours in
			$H$, then we are done by double counting the edges between $V(H)$
			and $V(G)\setminus V(H)$: there are at most $3(n-15)$ such edges,
			and by the minimum degree condition, more than $15(n/5 - 3) =
			3(n-15)$ such edges, a contradiction. Therefore, we may assume that
			there is a vertex $v$ of degree $4$ in $H$. By the preceding
			analysis, we may assume that $y_1z_1$ and $y_1w_3$ are edges: if
			not, replace $y_1$ by $v$.

			The proof breaks into two cases from here:
			\begin{enumerate}[(a)]
				\item \label{itm:ind-6-cycle-subcase1}
					$z_3$ and $w_1$ have a common neighbour.
				\item \label{itm:ind-6-cycle-subcase2}
					$z_3$ and $w_1$ do not have a common neighbour.
			\end{enumerate}
					
			Assuming (\ref{itm:ind-6-cycle-subcase1}), let $w$ be a common
			neighbour of $z_3$ and $w_1$, and denote by $H''$ the $10$-vertex
			graph induced on $V(H)\setminus V(P_{14})$ (see the black vertices in
			\Cref{fig:forbid-ind-6-cycle-marked2}).
				
			\begin{figure}[h]
				\centering
				\includegraphics[scale = .85]{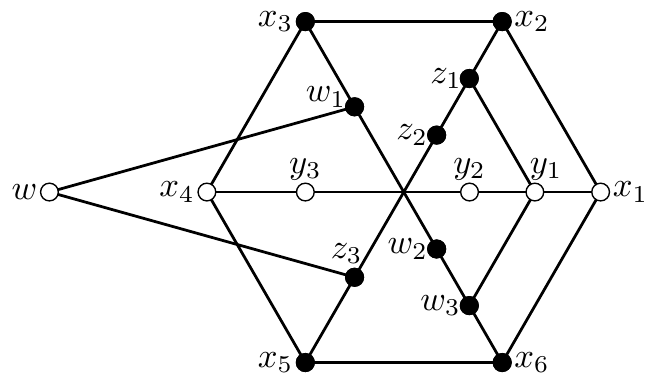}
				\caption{$H''$, with common neighbour $w$}
				\label{fig:forbid-ind-6-cycle-marked2}
			\end{figure}
		
			\begin{claim}\label{claim:ind-6-cycle-claim2}
				Every vertex of $G$ has at most two neighbours in $H''$.
			\end{claim}

			\begin{proof}
				The proof is similar to that of \Cref{claim:ind-6-cycle-claim1}.
				Any vertex $u$ is adjacent to at most two vertices of $C\cap
			   H''$.  If $u$ is joined to $x_2$ and $x_6$, then its only
				other possible neighbour is $y_1$, but $y_1 \notin H''$.
				A similar statement holds if $u$ is joined to $x_3$ and $x_5$.
				
				Suppose now that $u$ has precisely one neighbour in $C\cap
				H''$, and suppose this neighbour is $x_2$. By our preceding
				analysis of possible neighbourhoods in $H$, $u$'s only other
				possible neighbours are $y_1, z_2,$ and $w_1$. However, $y_1
				\notin H''$ and $u$ cannot be joined to both $z_2$ and $w_1$:
				otherwise, the $5$-cycle $(z_2z_3ww_1u)$ is formed. Hence, $u$
				is joined to at most two vertices of $H''$.
				
				Similarly, if $u$ is a vertex whose only neighbour in $C\cap
				H''$ is $x_3$, then $u$'s only other possible neighbours are
				$y_3, w_2,$ and $z_1$. But $y_3 \notin H''$ and $u$ cannot be
				joined to both $w_2$ and $z_1$: otherwise the $5$-cycle
				$(w_2w_3y_1z_1u)$ is created. The other cases (i.e.,~$u$ joined
				to $x_5$ or $x_6$) are symmetric.
				
				Finally, suppose $u$ is a vertex with no neighbour in $C\cap
				H''$.  If $u$ is joined to a middle vertex, say, without loss
				of generality, $w_2$, then by assumption $u$ cannot be joined to
				$z_2$. Hence the only other possible neighbours of $u$ are $z_1$
				and $z_3$. But $u$ cannot be joined to $z_1$, since otherwise
				$(uw_2w_3y_1z_1)$ is a $5$-cycle in $G$. If $u$ is not joined to
				a middle vertex, then observe that it can be adjacent to at most
				one vertex from each pair $\{z_3, w_3\}$ and $\{z_1, w_1\}$.
				
				This completes the proof of \Cref{claim:ind-6-cycle-claim2}.
			\end{proof}
			
			Let us now assume (\ref{itm:ind-6-cycle-subcase2}), that $z_3$ and
			$w_1$ do not have a common neighbour. Denote by $H'''$ the
			$10$-vertex graph induced on $V(H)\setminus \{x_1, x_2, x_6, y_2,
			y_3\}$ (see the black vertices in \Cref{fig:forbid-ind-6-cycle-marked3}).
			\begin{figure}[h]
				\centering
				\includegraphics[scale = .85]{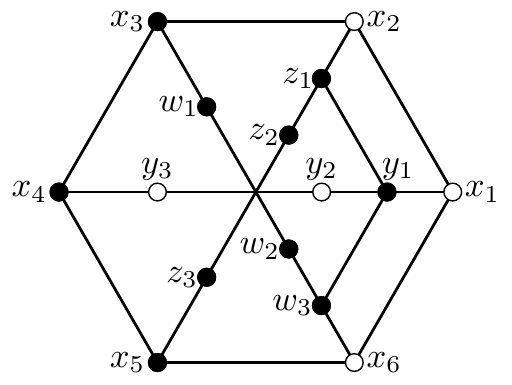}
				\caption{$H'''$}
				\label{fig:forbid-ind-6-cycle-marked3}
			\end{figure}   
			
			\begin{claim}\label{claim:ind-6-cycle-claim3}
				Every vertex of $G$ has at most two neighbours in $H'''$.
			\end{claim}
			
			\begin{proof}
				If a vertex $u$ of $G$ has two neighbours in $C\cap
				H'''$, then $u$ must be joined to $x_3$ and $x_5$. But $u$'s
				only other potential neighbour is $y_3$, and $y_3 \notin
				H'''$.
				
				Suppose $u$ is a vertex with exactly one neighbour in $C\cap
				H'''$.  First, suppose $u$ is joined to $x_3$. Then $u$'s only
				other possible neighbours are $w_2, y_3,$ and $z_1$. But $y_3
				\notin H'''$ and $u$ cannot be joined to both $z_1$ and $w_2$,
				as otherwise the $5$-cycle $(uw_2w_3y_1z_1)$ is in $G$.  The
				case when $u$ is joined to $x_5$ is dealt with symmetrically.
			
				Suppose now that $u$ is joined to $x_4$. The only other
				possible neighbours are then $y_2, z_3,$ and $w_1$. Observe
				that $y_2 \notin H'''$, and, by assumption, $z_3$ and $w_1$
				have no common neighbour, so $u$ is joined to at most one of
				them.
				
				One may (as in the proof of \Cref{claim:ind-6-cycle-claim2})
				dispense with the case when $u$ is a vertex with no neighbour
				in $C\cap H'''$.  Thus, no vertex of $G$ has more than
				two neighbours in $H'''$ and this completes the proof
				of \Cref{claim:ind-6-cycle-claim3}.
			\end{proof}
			
			We may now complete the proof of \Cref{lem:induced-cycle} in Case
			\ref{itm:ind-6-cycle-no-middle}. Indeed, if
			(\ref{itm:ind-6-cycle-subcase1}) holds, then apply
			\Cref{claim:ind-6-cycle-claim2} together with the usual double
			counting technique to produce a contradiction.  If instead
			(\ref{itm:ind-6-cycle-subcase2}) holds, then apply
			\Cref{claim:ind-6-cycle-claim3} together with double counting.
			This completes the proof of \Cref{lem:induced-cycle}. 	 
	\end{proof}

\section{$12$-cycles with few diagonals}\label{sec:12-cycles-few-chords}
    
	Our aim in this section is to prove \Cref{lem:cycle-chords}.  We divide the
	proof into steps, according to the number of diagonals.  Note that the case
	of having precisely five diagonals is immediate from
	\Cref{lem:induced-cycle} that forbids induced $6$-cycles.  It remains to
	examine the situation when there are either two, three or four
	diagonals present.
         
	\begin{prop} \label{prop:cycle-four-chords}
		\supposemax{}
		Then $G$ has no $12$-cycle with exactly four diagonals.
	\end{prop}

	\begin{proof}

		Suppose that $(x_1 \ldots x_{12})$ is a $12$-cycle with exactly four
		diagonals.  Let $H$ be the graph induced by $\{x_1, \ldots, x_{12}\}$.
		In light of \Cref{lem:induced-cycle}, $G$ has no induced $6$-cycle, so
		we may assume that the edges $x_1 x_7, x_2 x_8, x_3 x_9, x_4 x_{10}$
		are present in the graph and that $x_5 x_{11}, x_6 x_{12}$ are
		non-edges. In fact, it is easy to verify that the only edges in $H$ are
		the edges of the $12$-cycle and these four diagonals.

	 	\begin{figure}[h]
			\centering
			\includegraphics[scale = .8]{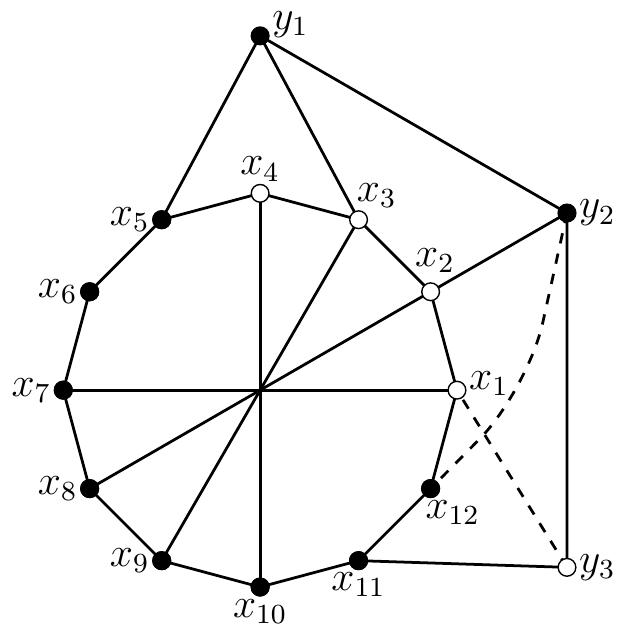}
			\caption{Constructing the graph $H'$ in the case of four diagonals}
			\label{fig:four-chords-added-vs-edges}
		\end{figure}
		
		The pair $\{x_5,  x_{11}\}$ is a non-edge in $G$, and so there is a
		path of length $2$ or $4$ between $x_5$ and $x_{11}$. In fact,
		the length must be $4$ because, otherwise, a cycle of length $3$
		or $5$ will be created.  Let $x_5 y_1 y_2 y_3 x_{11}$ be this
		$4$-path.  One may verify that $y_2 \notin H$, and possibly $y_3 =
		x_{12}$ or $y_1 = x_6$, but not both. We shall assume, without loss of
		generality, that $y_1 \neq x_6$.

		\begin{claim} \label{claim:edge-y1-x4}
			We may assume that $y_1 x_3$ and $y_2x_2$ are edges of $G$.
		\end{claim}

		\begin{proof}
			No two of the following vertices have a common neighbour: $x_3,
			x_6, x_9, x_{12}$ (they are at distance one or three from each
			other).  In other words, their neighbourhoods are pairwise
			disjoint, and so, by the minimum degree condition, every vertex in
			$G$ has a common neighbour with at least one of these four
			vertices.  Note that $y_2$ does not have a common neighbour with
			either $x_6$ or $x_{12}$ (this will create a $C_5$). By symmetry,
			we may assume that $y_2$ and $x_3$ have a common neighbour $u$.  If
			$u = y_1$, \Cref{claim:edge-y1-x4} follows.  Thus, we suppose
			otherwise.  Consider the $6$-cycle $(u y_2 y_1 x_5 x_4 x_3)$.
			Since there are no induced $6$-cycles, one of the following is an
			edge: $y_1 x_3, y_2 x_4, u x_5$.  If $y_1 x_3$ is an edge, the
			claim follows; $y_2 x_4$ cannot be an edge (because of the
			$5$-cycle $(y_2 x_4 x_{10} x_{11} y_3)$); if $u x_5$ is an edge, we
			replace $y_1$ by $u$ to prove the first part of the Claim.
			
			To see the second part, by considering the neighbours of $x_2, x_5,
			x_8, x_{11}$, we have that $y_3$ has a common neighbour with $x_2$
			or $x_8$.  If $u$ is a common neighbour of $y_3$ and $x_2$, we may
			assume that $u \neq y_2$ (otherwise, we are done).  By
			considering the $6$-cycle $(u x_2 x_3 y_1 y_2 y_3)$, either $y_2
			x_2$ or $u y_1$ is an edge.  We may assume that $u y_1$ is an edge.
			Then, by replacing $y_2$ by $u$ we obtain the required property.
			Now suppose that $y_3$ and $x_8$ have a common neighbour $u$. By
			considering $(u x_8 x_9 x_{10} x_{11} y_3)$, $u$ is adjacent to
			$x_{10}$. This, in turn, implies that $u$ is adjacent to $x_3$ (see
			$(u x_8 x_2 x_3 x_4 x_{10})$), a contradiction: the $5$-cycle $(u
			x_3 y_1 y_2 y_3)$ is formed.
		\end{proof}                
			
		Denote by $H'$ the graph induced by $\{x_5, \dots, x_{12}, y_1, y_2\}$
		(see the black vertices in \Cref{fig:four-chords-added-vs-edges}). We
		shall show that every vertex of $G$ has few neighbours in $H'$,
		yielding a contradiction to the minimum degree condition on $G$. More
		precisely, we have the following:

		\begin{claim} \label{claim:four-chords-few-neighs}
			Every vertex of $G$ has at most two neighbours in $H'$.
		\end{claim}
		\begin{proof}
			We first prove that $H$ is well-behaved. First, note that no vertex
			$u$ in $G$ can be adjacent to all of $\{x_4, x_6, x_{11}\}$.
			Indeed, otherwise, $(u x_{11} x_{12} x_1 x_7 x_6)$ is an induced
			$C_6$ (the addition of any chord to this cycle creates a triangle
			or a pentagon), contradicting \Cref{lem:induced-cycle}.  By
			symmetry, no vertex can be adjacent to all vertices in one of the
			following sets: $\{x_5, x_7, x_{12}\}$, $\{x_1, x_{6}, x_{11}\}$,
			$\{x_5, x_{10}, x_{12}\}$.  We conclude that no vertex can be
			adjacent to both $x_6$ and $x_{11}$. Indeed, by considering the
			$6$-cycle $(x_1 x_7x_6ux_{11}x_{12})$, since there is no induced
			$C_6$ in $G$, $u$ must be adjacent to $x_1$, contradicting the
			above.  Similarly, no vertex is adjacent to both $x_5$ and
			$x_{12}$.  One may check that any other possible neighbourhood of a
			vertex of $G$ in $H$ is contained in the neighbourhood of a vertex
			in $H$.
						   
			Now, as $H$ is well-behaved, no vertex in $G$ has more than two
			neighbours in $H' \cap H$.  Thus, if a vertex $u$ has three
			neighbours in $H'$, at least one of them is either $y_1$ or $y_2$.
			If $u$ is adjacent to $y_1$, then the only other neighbours $u$ can
			have in $H'$ are $x_6, x_9, x_{12}$, but no two of these vertices
			may have a common neighbour.  Similarly, if $u$ is adjacent to
			$y_2$, its other possible neighbours in $H'$ are $x_5,x_8, x_{11}$,
			no two of which have a common neighbour. The Claim follows.
		\end{proof}

		Using \Cref{claim:four-chords-few-neighs}, we may now finish the
		proof of \Cref{prop:cycle-four-chords} by double counting the number
		of edges between $V(H')$ and $V(G) \setminus V(H')$, as usual.
	\end{proof}

	Now we deal with the remaining case, of a $12$-cycle with two or three
	diagonals, and thereby complete the proof of \Cref{lem:cycle-chords}.

    \begin{prop} \label{prop:cycle-two-three-chords}
        \supposemax{}
		Then $G$ induces no $12$-cycle with two consecutive diagonals and
		at most one additional chord.
    \end{prop}

    \begin{proof}
		Suppose that $C = (x_1 \ldots x_{12})$ is a $12$-cycle with two
		consecutive diagonals $x_1 x_7$ and $x_2 x_8$, and at most one
		additional chord.  We note that any additional chord (that does not
		complete a triangle or $5$-cycle) is a diagonal in one of the following
		$12$-cycles $(x_1 \ldots x_{12})$ or $(x_2 \ldots x_7 x_1 x_{12} \ldots
		x_8)$, both of which have two consecutive diagonals.  Hence, and by
		symmetry, we may assume that the additional chord is either $x_6
		x_{12}$ or $x_5 x_{11}$.  However, if $x_5 x_{11}$ is the additional
		chord, then $(x_1 x_7 x_6 x_5 x_{11} x_{12})$ is an induced $6$-cycle,
		contradicting \Cref{lem:induced-cycle}.  Thus we assume that, if there
		is an additional chord, it is $x_6 x_{12}$. Furthermore, if $x_6x_{12}$
		is not an edge, we assume that $G$ contains no $12$-cycles with two
		consecutive diagonals and exactly one extra chord.  Let $H$ be the
		graph induced on $\{x_1, \ldots, x_{12}\}$ and denote $H' = H \setminus
		\{x_1, x_7\}$.
        
        \begin{claim} \label{claim:two-three-chords-H'}
            Every vertex in $G$ has at most two neighbours in $H'$.
        \end{claim}

        \begin{proof}
           
			Suppose that $u$ has three neighbours in $H'$.  It follows by
			symmetry that $u$ has two neighbours in $\{x_2, \ldots, x_6\}$,
			which we can denote by $x_{i - 1}$ and $x_{i + 1}$ for some $i \in
			\{3, 4, 5\}$ (by \Cref{prop:well-behaved}), and another neighbour
			$x_j$ for some $j \in \{8, \ldots, 12\}$.  But then, by replacing
			$x_i$ by $u$, we may assume that $x_i$ is joined to $x_j$. This is
			a contradiction: either to \Cref{prop:cycle-four-chords} (if $C$
			had three chords, i.e.~if $x_6 x_{12}$ is an edge, then now it has
			four chords); or, if $x_6 x_{12}$ is not an edge, to the assumption
			that there is no $12$-cycle with two consecutive diagonals and an
			additional chord.
        \end{proof}

		\Cref{prop:cycle-two-three-chords} follows from
		\Cref{claim:two-three-chords-H'} by double counting the number of edges
		between $V(H')$ and $V(G) \setminus V(H')$. The proof of
		\Cref{lem:cycle-chords} is therefore complete.
    \end{proof}
    
\section{Two $7$-cycles intersecting in a $3$-path}\label{sec:7-cycles-3-path}

    \begin{figure}[h]
        \centering
        \includegraphics[scale = 1]{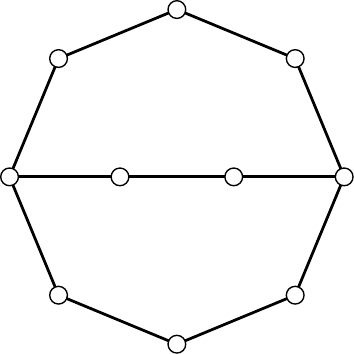}
        \caption{two $7$-cycles intersecting in a $3$-path}
        \label{fig:two-cycles}
    \end{figure}
	In this section we prove \Cref{lem:two-cycles}; that is, the graph in
	\Cref{fig:two-cycles} cannot appear as an induced subgraph of a maximal
	\cfree{} graph on $n$ vertices with minimum degree larger than $n/5$.
   
    \begin{proof}[ of \Cref{lem:two-cycles}]
		Suppose that $H$ is an induced subgraph of $G$ which is the union of
		two $7$-cycles intersecting in a path of length $3$.  Denote the two
		$7$-cycles by $(x_1 x_2 x_3 x_4 x_5 x_6 x_7)$ and $(x_1 x_2 x_3 x_4 x_8
		x_9 x_{10})$ (see \Cref{fig:two-cycles}).  We start by showing that $H$
		is a \wellbehaved{} subgraph of $G$ (recall \Cref{def:well-behaved}), a
		fact that will be useful in the proof.
         
        \begin{claim} \label{claim:two-cycles-well-behaved}
            The graph $H$ is \wellbehaved.
        \end{claim}

        \begin{proof}
			If $H$ is not \wellbehaved{}, then, up to relabelling, one of the
			following two pairs has a common neighbour in $G$: $\{x_6, x_9\}$
			or $\{x_5, x_{10}\}$.  If $u$ is a neighbour of $x_6$ and $x_9$
			then, by \Cref{lem:induced-cycle}, $u$ is also a neighbour of $x_1$
			(consider the $6$-cycle $(u x_6 x_7 x_1 x_{10} x_9)$), and of $x_4$
			(consider the $6$-cycle $(ux_6x_5x_4x_8x_9)$). But this produces
			the $5$-cycle $(ux_1x_2x_3x_4)$.  Now suppose that $u$ is a
			neighbour of both $x_5$ and $x_{10}$.  By considering the $6$-cycle
			$(u x_5 x_4 x_8 x_9 x_{10})$, $u$ must be adjacent to $x_8$.  Now
			consider the $7$-cycle $(u x_{10} x_1 x_2 x_3 x_4 x_8)$.  The
			vertex $x_6$ has no neighbour in $C$ ($x_6$ cannot be adjacent to
			$u$), but $x_5$ has two neighbours in $C$ ($x_4$ and $u$). This is
			a contradiction to \Cref{cor:vx-cycle-two-neighs}.
        \end{proof}
		Arguments as in \Cref{claim:two-cycles-well-behaved}, using
		\Cref{cor:vx-cycle-two-neighs} and \Cref{lem:induced-cycle} will appear
		frequently in the proof of \Cref{lem:two-cycles}.

		Since $x_6$ and $x_8$ are nonadjacent, there is a $4$-path with ends
		$x_6$ and $x_8$ (a $2$-path would create a $C_5$).  Up to relabelling,
		three cases arise:
        \begin{enumerate}
            \item \label{itm:three-path-vertical}
                There is a $3$-path $x_6 y_1 y_2 x_9$ between $x_6$ and
                $x_9$. The vertices $y_1$ and $y_2$ are not in $H$.
            \item \label{itm:three-path-diagonal}
                There is a $3$-path $x_7 y_1 y_2 x_8$ between $x_7$ and
                $x_8$. The vertices $y_1$ and $y_2$ are not in $H$.
            \item \label{itm:four-path}
                There is a $4$-path $x_6 y_1 y_2 y_3 x_8$ between $x_6$ and
                $x_8$. The vertices $y_1, y_2, y_3$ are not in $H$.
        \end{enumerate}

        In the rest of the proof, we show that each of the three cases is
        impossible, thus completing the proof of \Cref{lem:two-cycles}. Case
        \ref{itm:three-path-diagonal} will be the most difficult to resolve.

		\subsection{Case \ref{itm:three-path-vertical}: \normalfont a $3$-path
			between $x_6$ and $x_9$}

            Denote by $H'$ the graph induced by $\{x_1, \dots, x_{10}, y_1,
            y_2\}$.

            \begin{claim} \label{claim:case-one-well-behaved}
                $H'$ is \wellbehaved.
            \end{claim}

            \begin{proof}
				Suppose that $H'$ is not \wellbehaved.  Up to relabelling, it
				follows that $y_1$ and $x_3$ have a common neighbour $u$
				(recall that $H$ is \wellbehaved).  By considering the
				$6$-cycle $(u x_3 x_4 x_5 x_6 y_1)$ and in light of
				\Cref{lem:induced-cycle}, it follows that $u$ is adjacent to
				$x_5$.  Consider the $7$-cycle $C = (x_1 x_2 x_3 u x_5 x_6
				x_7)$. Observe that $x_4$ has two neighbours in $C$ ($x_3$ and
				$x_5$), but $x_8$ has no neighbour in $C$ ($x_8$ cannot be
				adjacent to $u$).  This contradicts
				\Cref{cor:vx-cycle-two-neighs}.
            \end{proof}

			Now consider the graph $H'' = H' \setminus \{x_5, x_{10}\}$. It
			follows from \Cref{claim:case-one-well-behaved} that every vertex
			in $G$ has at most two neighbours in $H''$. The usual argument, of
			double counting the edges between $V(H'')$ and $V(G) \setminus
			V(H'')$, leads to a contradiction to the minimum degree condition,
			thus completing the proof of \Cref{lem:two-cycles} in Case
			\ref{itm:three-path-vertical}.

        \subsection{Case \ref{itm:three-path-diagonal}: \normalfont a $3$-path
            between $x_7$ and $x_8$}

            Denote by $H'$ the graph induced by $\{x_1, \dots, x_{10}, y_1,
            y_2\}$ (see \Cref{fig:two-cycles-case2}).
            
            \begin{figure}[h]
                \centering
                \includegraphics[scale = 1]{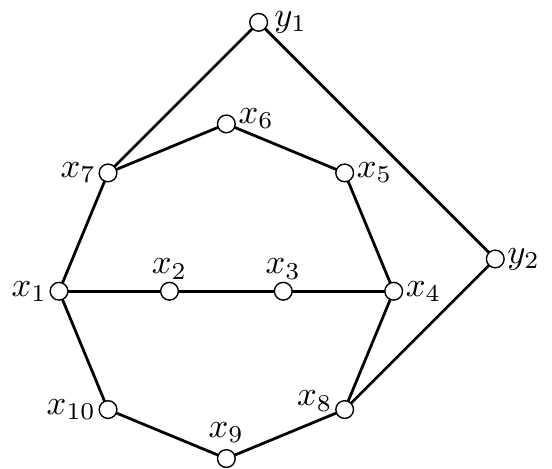}
                \caption{Case 2: a path of length $3$ between $x_7$ and $x_8$}
                \label{fig:two-cycles-case2}
            \end{figure}
            
            \begin{claim} \label{claim:case-two-well-behaved}
                The graph $H'$ is \wellbehaved.
            \end{claim}
            \begin{proof}
				If $H'$ is not \wellbehaved, then up to relabelling, $y_1$ and
				$x_3$ have a common neighbour $u$ (recall that $H$ is
				\wellbehaved{} by \Cref{claim:two-cycles-well-behaved}).
				Consider the $6$-cycle $(u y_1 x_7 x_1 x_2 x_3)$.  Since there
				is no induced $6$-cycle (\Cref{lem:induced-cycle}), either
				$y_1$ is adjacent to $x_2$, or $u$ is adjacent to $x_1$.  The
				former case leads to a contradiction similarly to
				\Cref{claim:case-one-well-behaved}: then $x_1$ has two
				neighbours in the $7$-cycle $(x_2 x_3x_4x_5x_6x_7y_1)$ whereas
				its neighbour $x_{10}$ has no neighbour there, contradicting
				\Cref{cor:vx-cycle-two-neighs}.  So, suppose the latter case
				holds, i.e.~$u$ is adjacent to $x_1$. But then $u$ has two
				neighbours in the $7$-cycle $(x_1 x_2 x_3 x_4 x_8 x_9 x_{10})$
				whereas $y_1$ has none, a contradiction.
            \end{proof}

			As before, in light of the missing edge $x_6 x_{10}$, one of the
			following three cases holds.
            \begin{enumerate}[(a)]
                \item \label{itm:two-cycles-case-a}
                    There is a $3$-path $x_6 z_1 z_2 x_9$ between $x_6$ and
                    $x_9$.
                \item \label{itm:two-cycles-case-b}
                    There is a $3$-path $x_5 z_1 z_2 x_{10}$ between $x_5$ and
                    $x_{10}$.
                \item \label{itm:two-cycles-case-c}
                    There is a $4$-path $x_6 z_1 z_2 z_3 x_{10}$ between $x_6$
                    and $x_{10}$.
            \end{enumerate}
            However, (\ref{itm:two-cycles-case-a}) does not hold, as we have
            seen in the previous subsection. So it remains to consider
            (\ref{itm:two-cycles-case-b}) and (\ref{itm:two-cycles-case-c}).

            \subsubsection*{Case
                \ref{itm:three-path-diagonal}\ref{itm:two-cycles-case-b}:
                \normalfont $3$-paths between $x_7$ and $x_8$ and between $x_5$
                and $x_{10}$}

				Denote by $F$ the graph induced by
				$\{x_1,\ldots,x_{10},y_1,y_2,z_1,z_2\}$ (see
				\Cref{fig:two-cycles-case-2b}). It is easy to check that the
				vertices $y_1, y_2, z_1, z_2$ are distinct.  Define $F' = F
				\setminus \{x_1, x_4, x_7, x_8\}$ (see
				\Cref{fig:two-cycles-case-2b}).
                
                \begin{figure}[h]
                    \centering
                    \includegraphics[scale = 1]{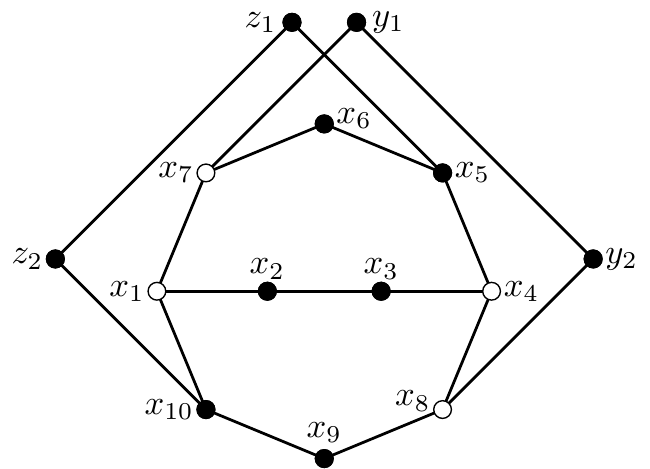}
                    \caption{Case 2\protect\ref{itm:two-cycles-case-b}: the graphs $F$ and
                    $F'$ (marked in black)}
                    \label{fig:two-cycles-case-2b}
                \end{figure}

                \begin{claim} \label{claim:two-neighs-case-2a}
					Every vertex of $G$ has at most two neighbours in $F'$.
                \end{claim}

                \begin{proof}
					Suppose that there is a vertex $u$ with at least three
					neighbours in $F'$.  We note that $u$ is adjacent to one of
					$y_1$ and $y_2$ and also to one of $z_1$ and $z_2$. Indeed,
					otherwise, it is easy to check that $u$ has at most two
					neighbours in $F'$ using the fact that $H$ is
					\wellbehaved{} (and thus also the graph induced by
					$\{x_1,...,x_{10},z_1,z_2\}$; see
					\Cref{claim:two-cycles-well-behaved}).  By symmetry, we may
					assume that $u$ is adjacent to $y_1$.  Suppose that $u$ is
					also adjacent to $z_1$. By considering the $6$-cycle $(u
					z_1 x_5 x_6 x_7 y_1)$, it follows that $u$ is adjacent to
					$x_6$. But, then, $x_7$ has two neighbours in the $7$-cycle
					$(u y_1 y_2 x_8 x_4 x_5 x_6)$ while $x_1$ has none.  This
					is a contradiction to \Cref{cor:vx-cycle-two-neighs}.

					It remains to consider the case where $u$ is adjacent to
					both $y_1$ and $z_2$. It follows that $u$ is adjacent to
					$x_1$ (consider $(u y_1 x_7 x_1 x_{10} z_2)$). This is a
					contradiction to \Cref{cor:vx-cycle-two-neighs}: $x_{10}$
					has two neighbours in $(u z_2 z_1 x_5 x_6 x_7 x_1)$ whereas
					$x_9$ has none. 
                \end{proof}

                \Cref{claim:two-neighs-case-2a} leads to a contradiction by
                double counting the edges between $V(F')$ and $V(G) \setminus
                V(F')$.  This completes the proof of \Cref{lem:two-cycles} in
                this case.

		\subsubsection*{Case 
			\ref{itm:three-path-diagonal}\ref{itm:two-cycles-case-c}:
			\normalfont a $3$-path between $x_7$ and $x_8$ and a $4$-path
			between $x_6$ and $x_{10}$}
		
			Denote by $F$ the graph induced by $\{x_1,\ldots,x_{10},y_1,y_2,z_1,
			z_2,z_3\}$ (see \Cref{fig:two-cycles-case-2c}).

			\begin{figure}[h]
				\centering
				\includegraphics[scale = 1]{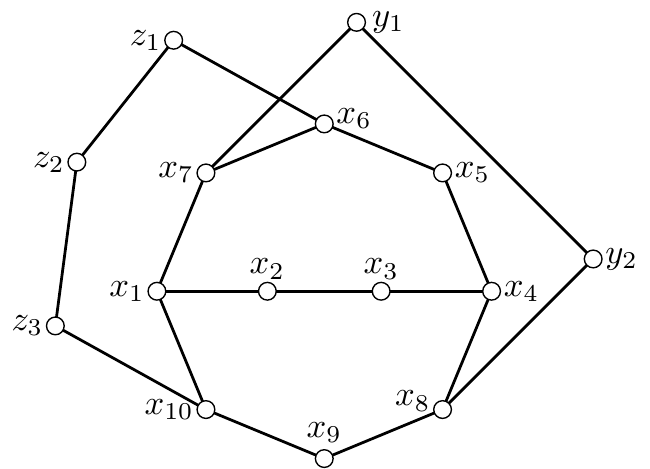}
				\caption{Case
					\protect\ref{itm:three-path-diagonal}\protect\ref{itm:two-cycles-case-c}: the
					graph $F$}
				\label{fig:two-cycles-case-2c}
			\end{figure}

			\begin{claim} \label{claim:no-additional-edges-case2c}
				The only edges spanned by $F$ are those spanned by $H$ and the
				edges of the two paths $x_6 z_1 z_2 z_3 x_{10}$ and $x_7 y_1
				y_2 x_8$.
			\end{claim}

			\begin{proof}
				First we note that $y_1$ and $y_2$ do not have additional
				neighbours in $\{x_1,\ldots, x_{10}\}$.  Indeed, by symmetry we
				assume that $y_1$ has an additional neighbour in $H$. The only
				possible such neighbour is $x_2$.  We reach a contradiction
				to \Cref{cor:vx-cycle-two-neighs} (consider the $7$-cycle $(y_1
				x_2 x_3 x_4 x_5 x_6 x_7)$ and the vertices $x_1$ and
				$x_{10}$).

				We now show that $z_1$, $z_2$ and $z_3$ do not send additional
				edges into $H$.  Using the fact that $H'$ is \wellbehaved{},
				the only possible additional neighbour of $z_1$ is $x_4$.  But
				then, by replacing $x_5$ by $z_1$, we may assume that there
				is a $3$-path from $x_5$ to $x_{10}$. This leads to a
				contradiction, as we have seen in Case
				\ref{itm:three-path-diagonal}\ref{itm:two-cycles-case-b}.
				Similarly, the possible additional neighbours of $z_3$ in $H$
				are $x_8$ and $x_2$. If $z_3$ is adjacent to $x_8$ then, by
				replacing $x_9$ by $z_3$, we may assume that there is a
				$3$-path between $x_6$ and $x_9$, contradicting Case
				\ref{itm:three-path-vertical}.  If $z_3$ is adjacent to $x_2$
				we reach a contradiction to \Cref{cor:vx-cycle-two-neighs}
				($x_1$ has two neighbours in the $7$-cycle $(z_3 x_2 x_3 x_4
				x_8 x_9 x_{10})$ while $x_7$ has none).  The possible
				neighbours of $z_2$ in $H$ are $x_3$, $x_5$ and $x_9$. But
				$z_2$ is not adjacent to $x_5$ or $x_9$, because, otherwise,
				there is a $3$-path between $x_5$ and $x_{10}$ or between $x_6$
				and $x_9$, contradicting previous cases. Furthermore, $z_2$ is
				not adjacent to $x_3$ because, otherwise, $(z_1 z_2 x_3 x_4
				x_5 x_6)$ is an induced $6$-cycle, contradicting
				\Cref{lem:induced-cycle}.

				Finally, we show that there are no edges between $\{z_1, z_2,
				z_3\}$ and $\{y_1, y_2\}$. The only such edges that do not
				create a triangle or pentagon are $z_1 y_1$ and $z_2 y_2$.  If
				$z_1$ is adjacent to $y_1$ we reach a contradiction to
				\Cref{cor:vx-cycle-two-neighs} (see $(z_1 y_1 y_2 x_8 x_4 x_5
				x_6)$ and the vertices $x_1$, $x_7$), and if $z_2 y_2$ is an
				edge, a contradiction to \Cref{lem:induced-cycle} is reached
				(consider the induced $6$-cycle $(z_1 z_2 y_2 y_1 x_7 x_6)$).
				This completes the proof of
				\Cref{claim:no-additional-edges-case2c}.
			\end{proof}
			
			The following claim states that no vertex has more than three
			neighbours in $F$. Since $|V(F)| = 15$, this is a contradiction to
			the minimum degree condition on $G$ by the usual double counting
			argument, hence the proof of \Cref{lem:two-cycles} in this case
			follows.
			\begin{claim} 
				No vertex has more than three neighbours in $F$.
			\end{claim}

			\begin{proof}
				Since $H'$ is \wellbehaved{} (see
				\Cref{claim:case-two-well-behaved}; recall that $H'$ is the
				graph induced by the set $\{x_1,\ldots,x_{10}, y_1, y_2\}$) and
				has maximum degree $3$, if there is a vertex $u$ with four
				neighbours in $F$, it must be adjacent to at least one of
				$z_1, z_2, z_3$.  We note that $u$ cannot be adjacent to both
				$z_1$ and $z_3$ because then, by replacing $z_2$ by $u$, we may
				assume that $z_2$ has an additional edge in $F$, a
				contradiction to \Cref{claim:no-additional-edges-case2c}.  It
				follows that $u$ has one neighbour among $z_1, z_2, z_3$ and at
				least three neighbours in $H'$. Since $H'$ is \wellbehaved{},
				$u$ is adjacent to all three neighbours of a vertex $v$ in
				$H'$ of degree three (in $H'$). But then, by replacing $v$ by
				$u$, we may assume that $v$ has an additional edge in $F$, a
				contradiction to \Cref{claim:no-additional-edges-case2c}.
			\end{proof}

    \subsection{Case \ref{itm:four-path}: \normalfont a $4$-path between
        $x_6$ and $x_8$}

		Denote by $H'$ the graph induced by $\{x_1, \dots, x_{10}, y_1, y_2,
		y_3\}$, and let $H'' = H' \setminus \{x_5, x_7, y_3\}$ (see
		\Cref{fig:four-path}).

        \begin{figure}[h]
            \centering
            \includegraphics[scale = 1]{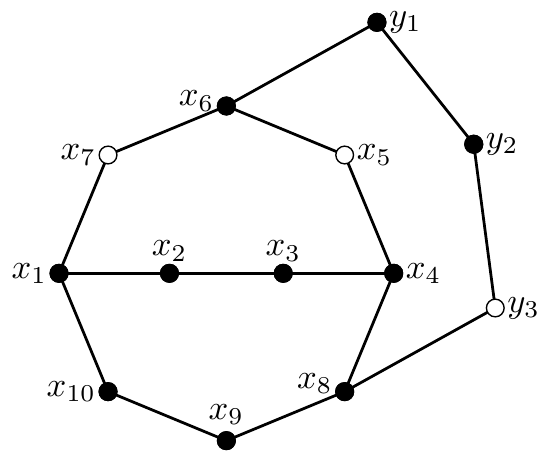}
            \caption{Case \protect\ref{itm:four-path}: the graphs $H'$ and $H''$ (marked in black)} 
            \label{fig:four-path}
        \end{figure}

        \begin{claim} \label{claim:no-additional-edges}
            The only edges in $H'$ are those spanned by $H$ or by the path
            $x_6 y_1 y_2 y_3 x_8$.
        \end{claim}
        \begin{proof}
			Suppose that there are additional edges.  These must be between
			$\{y_1, y_2, y_3\}$ and $V(H)$.   The only possible neighbour (that
			is not already accounted for) of $y_1$ in $H'$ is $x_1$. But then,
			by replacing $x_7$ by $y_1$, we reach a contradiction to Case
			\ref{itm:three-path-diagonal}.
            
			The only possible additional neighbours of $y_3$ in $H$ are $x_3$
			and $x_{10}$. If $y_3$ is adjacent to $x_3$, then $x_4$ has two
			neighbours in $(x_1 x_2 x_3 y_3 x_8 x_9 x_{10})$ whereas $x_5$ has
			none, a contradiction to \Cref{cor:vx-cycle-two-neighs}.  If $y_3$
			is adjacent to $x_{10}$ then, by replacing $x_9$ with $y_3$, we
			reduce to Case \ref{itm:three-path-vertical}.

			The only possible additional neighbours of $y_2$ in $H$ are $x_2,
			x_7, x_9$.  If $y_2$ is adjacent to $x_9$ or $x_7$, we reduce to
			Case~\ref{itm:three-path-vertical}
			or~\ref{itm:three-path-diagonal}, respectively.  Finally, if $y_2$
			is adjacent to $x_2$ then $(x_6 x_7 x_1 x_2 y_2 y_1)$ is an induced
			$6$-cycle, a contradiction to \Cref{lem:induced-cycle}.
        \end{proof}

        \begin{claim} \label{claim:four-path-no-three-neighs}
            No vertex in $G$ has more than two neighbours in $H''$.
        \end{claim}

        \begin{proof}
			Suppose that there is a vertex $u$ in $G$ with three neighbours in
			$H''$.  Since $H$ is \wellbehaved{} (see
			\Cref{claim:two-cycles-well-behaved}), $u$ must be a neighbour of
			either $y_1$ or $y_2$.

			Suppose first that $u$ is a neighbour of $y_1$.  The other possible
			neighbours of $u$ in $H''$ are $x_2, x_3, x_9, x_{10}$.  Out of
			these four vertices, the only two that may have a common neighbour
			are $x_2$ and $x_{10}$.  By considering the $6$-cycle $(u x_2 x_1
			x_7 x_6 y_1)$, it follows that $u$ is adjacent also to $x_7$,
			i.e.~$u$ is adjacent to $x_2, x_7, x_{10}, y_1$.  By replacing
			$x_1$ by $u$, we may assume that $y_1$ is adjacent to $x_1$, a
			contradiction to \Cref{claim:no-additional-edges}.
            
			We may now assume that $u$ is adjacent to $y_2$.  The other
			possible neighbours of $u$ in $H''$ are $x_1, x_2, x_3, x_6, x_8,
			x_{10}$.  If $u$ is adjacent to $x_6$ or $x_8$, then by replacing
			$y_1$ or $y_3$ by $u$ we see that $u$ cannot have any additional
			neighbours in $H''$: otherwise we reach a contradiction to
			\Cref{claim:no-additional-edges}.  It follows that $u$ is not
			adjacent to $x_1$, because otherwise, $(u x_1 x_7 x_6 y_1 y_2)$ is
			an induced $6$-cycle. Similarly, $u$ is not adjacent to $x_{10}$
			(see $(x_{10} x_9 x_8 y_3 y_2 u)$).  This completes the proof of
			\Cref{claim:four-path-no-three-neighs}, since the only remaining
			possible neighbours of $u$ are $x_2$ and $x_3$, and these do not
			have a common neighbour.
        \end{proof}

		By \Cref{claim:four-path-no-three-neighs}, we reach a contradiction
		using the usual double counting argument.  This completes the proof of
		\Cref{lem:two-cycles}.
    \end{proof}
    
\section{The proof of \Cref{thm:hom-to-family}}\label{sec:main-theorem}

	In this section we shall finish the proof of \Cref{thm:hom-to-family} by
	combining \Cref{thm:vx-cycle} along with some facts we have obtained
	regarding forbidden substructures in maximal \cfree{} graphs of large
	minimum degree. Recall that $F_k$ is the graph obtained from a $(5k -
	3)$-cycle (an edge, when $k = 1$) by adding all chords joining vertices at
	distances along the cycle of the form $5j +1$ for $j= 1, \ldots , k- 2$.
	First, we prove the following proposition, which records several useful
	properties of the graphs $F_k$ that we shall need in the sequel.
	
    \begin{prop} \label{prop:Fk}
        The following properties of $F_k$ hold.
		\begin{enumerate}[(a)]
            \item \label{itm:F-k-two-vs-cycle}
                Every two distinct vertices in $F_k$ ($k \ge 2$) are contained in a
                $7$-cycle.
            \item \label{itm:F-k-paths-between-vs}
                Let $x$ and $y$ be distinct vertices in $F_k$. Then there is a
                path of length $1$, $3$ or $5$ between $x$ and $y$.
            \item \label{itm:F-k-num-neighs}
                Let $F$ be a copy of $F_k$ in a maximal \cfree{} graph $G$ with
                $\delta(G) > n/5$. Then every vertex in $G$ has either $k - 1$
                or $k$ neighbours in $F$.
            \item \label{itm:F-k-neighbd}
                Let $F$ be a copy of $F_k$ in a maximal \cfree{} graph $G$ 
                with $\delta(G) > n/5$. Denote the
                vertices of $F$ by $x_1, \ldots, x_{5k - 3}$ and its edges by the pairs
                $x_i x_j$ for which $|i - j| \equiv   1 \pmod{5}$.
                
                Then for every vertex $u$ in $G$ there is a vertex $x_i$ in $F$
                such that the neighbours of $u$ in $F$ are the neighbours of
                $x_i$ in $F$, except at most one of $x_{i - 1}$ and $x_{i + 1}$.
                In particular, $F$ is \wellbehaved{} as a subgraph of $G$.
        \end{enumerate}
    \end{prop}

    \begin{proof}

		To see (\ref{itm:F-k-two-vs-cycle}), denote the vertices and edges of
		$F_k$ as above (see (\ref{itm:F-k-neighbd})). Note that
		(\ref{itm:F-k-two-vs-cycle}) holds for $k = 2$.  Now suppose that $k
		\ge 3$ and let $x_i$ and $x_j$ be two distinct vertices in $F_k$.
		Suppose that $i < j$. If $j \le i + 6$, the two vertices are in the
		$7$-cycle $(x_i \ldots x_{i + 6})$.  Otherwise, $x_i$ and $x_j$ are two
		vertices in the graph induced by $F_k \setminus \{x_{i + 1}, \ldots,
		x_{i + 6}\}$ which is a copy of $F_{k - 1}$. Then, by induction, $x_i$
		and $x_j$ are in a copy of a $7$-cycle in $F_k$. Next, observe
		that (\ref{itm:F-k-paths-between-vs}) follows immediately from
		(\ref{itm:F-k-two-vs-cycle}).

		We prove (\ref{itm:F-k-num-neighs}) by induction on $k$.  For $F_1 =
		K_2$ the result is clear, and for $F_2 = C_7$ the result follows from
		\Cref{thm:vx-cycle}.  So suppose that $k\geq 3$ and the result holds
		for smaller values of $k$. Let $F$ be a copy of $F_k$ in $G$ as in the
		statement of (\ref{itm:F-k-num-neighs}), denote its vertices and edges
		as before, and let $u$ be a vertex of $G$.  Assume first that $u$ has
		$k+1$ neighbours in $F$. If $u$ has at most one neighbour in some
		consecutive interval $x_i,\ldots , x_{i+4}$ of five vertices, then $u$
		has at least $k$ neighbours in the copy of $F_{k-1}$ induced on
		$F\setminus \{x_i, \ldots , x_{i+4}\}$, a contradiction to the
		induction hypothesis.  Therefore, $u$ has at least two neighbours in
		every consecutive interval of five vertices. Suppose, without loss of
		generality, that $u$ is adjacent to $x_1$.  Then $u$ has at least
		$1+2(k-2) \ge k$ neighbours (recall that $k\geq 3$) in the copy of
		$F_{k-1}$ induced on $F\setminus \{x_{5k-7}, \ldots , x_{5k-3}\}$, a
		contradiction. If $u$ has at most $k-2$ neighbours in $F$, one of which
		is, say, $x_1$, then $u$ has at most $k-3$ neighbours in the copy of
		$F_{k-1}$ induced on $F\setminus \{x_1,\ldots , x_5\}$, contradicting
		the induction hypothesis.  It follows that $u$ has either $k-1$ or $k$
		neighbours in $F$, as required. 
        
		Finally, let us prove (\ref{itm:F-k-neighbd}). Let $F$ and $G$ be as in
		the statement of (\ref{itm:F-k-neighbd}), and suppose that $u$ has
		$k-1$ neighbours in $F$.  Then there must exist five consecutive
		vertices $x_\ell, \ldots , x_{\ell+4}$ which are not neighbours of $u$.
		Let $F'$ be the copy of $F_{k-1}$  given by $F\setminus \{x_\ell,
		\ldots , x_{\ell+4}\}$.  Then by induction there is a vertex $x$ of
		$F'$ such that $u$ is joined to all neighbours of $x$ in $F'$. We
		claim that $x=x_{\ell-1}$ or $x=x_{\ell+5}$. Indeed, note that $x$
		must be adjacent to precisely one of $x_{\ell-1}, x_{\ell+5}$ (it
		cannot be adjacent to both); otherwise, $u$ has no neighbour in the
		$7$-cycle $(x_{\ell-1}x_\ell \ldots x_{\ell +5})$, contradicting
		\Cref{thm:vx-cycle}. Suppose, without loss of generality, that $x$
		is joined to $x_{\ell-1}$.  Since $u$ must have a neighbour in the
		$7$-cycle $(x_\ell x_{\ell+1}\ldots x_{\ell+6})$, $u$ is also
		adjacent to $x_{\ell + 6}$.  It follows that $x = x_{\ell + 5}$ and
		$u$ has $k-1$ neighbours in $F$, which are precisely the neighbours
		of $x_{\ell+5}$, except for $x_{\ell+4}$.  Now, suppose that $u$
		has precisely $k$ neighbours in $F$. Then we may find two
		neighbours of $u$ that are at distance at most four. We claim that
		this implies there must be two neighbours at distance two apart.
		Indeed, they cannot be at distance three (this would produce a
		$5$-cycle). So suppose these neighbours are at distance four and
		suppose they are $x_i$ and $x_{i+4}$. Then
		$(ux_{i+4}x_{i+5}x_{i+6}x_i)$ is a $5$-cycle in $G$, a
		contradiction.  Accordingly, we may assume without loss of
		generality that $u$ is adjacent to both $x_2$ and $x_{5k-3}$.
		Consider the copy of $F_{k-1}$ given by $F\setminus \{x_3,\ldots ,
		x_7\}$ and apply induction. Clearly, we must have $u$ joined to
		$x_7$ ($u$'s only possible neighbour in $\{x_3,\ldots , x_7\}$) and
		the neighbourhood of $u$ in $F$ is precisely the neighbourhood of
		$x_1$ in $F$. This completes the proof of (\ref{itm:F-k-neighbd}).   
    \end{proof}
		
	We actually prove the followinbg theorem, which clearly implies
	\Cref{thm:hom-to-family}. It is the odd-girth $7$ analogue of a result of
	Chen, Jin, and Koh~\cite{chen-jin-koh} concerning triangle-free graphs of
	large minimum degree.
	
    \begin{thm}\label{thm:hom-or-Fk}
        \supposemax{}
		For every integer $k\geq 2$, if $G$ contains no copy of $F_k$, then $G$
		is homomorphic to $F_{k-1}$.	
    \end{thm}
    
    \begin{proof}
		We shall use induction on $k$. For $k=2$, it is easy to show that if
		$G$ contains no $7$-cycle, then it must be bipartite. So fix $k\geq 3$
		and suppose the result holds for smaller values of $k$. Let $G$ be as
		in the statement of the theorem and suppose it contains no copy of
		$F_k$. If $G$ contains no copy of $F_{k-1}$, then by induction $G$ is
		homomorphic to $F_{k-2}$. But $F_{k-1}$ contains $F_{k-2}$, so we are
		done. Hence we may assume that $G$ contains a copy of $F_{k-1}$. Let
		$H$ be a vertex-maximal blow-up of $F_{k-1}$ in $G$ with vertex classes
		$X_1, \ldots , X_{5k-8}$, where the edges of $H$ are $X_i-X_j$ edges
		for which $|i-j| \equiv 1\pmod{5}$. Our aim is to show that $G$ is a
		blow-up of $F_{k - 1}$, or, in other words, that $H$ spans all vertices
		in $G$.  Note that by (\ref{itm:F-k-num-neighs}) of \Cref{prop:Fk},
		every vertex in $V(G)\setminus V(H)$ has at most $k-1$ neighbours in
		$F_{k-1}$. 
        
		Suppose $u \in V(G)\setminus V(H)$ is adjacent to vertices in precisely
		$k-1$ of the classes of $H$. Without loss of generality, by
		(\ref{itm:F-k-neighbd}), we may assume that these classes are those in
		the neighbourhood of vertices in $X_1$, i.e.,~$X_2, X_7, \ldots ,
		X_{5k-8}$, and let $J = \{2, 7, \ldots, 5k-8\}$ be the set of indices
		$j$ such that $u$ has a neighbour in $X_j$.  We claim that $u$ must be
		adjacent to every vertex in each of these classes, contradicting the
		assumption that $H$ is a vertex-maximal blow-up in $G$.  Suppose this
		is not the case.  By (\ref{itm:F-k-num-neighs}), $u$ has a
		non-neighbour in at most one of the sets $X_j$ with $j \in J$ (indeed,
		otherwise we find a copy of $F_{k - 1}$ in which $u$ has at most $k -
		3$ neighbours).  Furthermore, by (\ref{itm:F-k-neighbd}), we may assume
		that this set is $X_2$. Let $y \in X_2$ be a neighbour of $u$ and let
		$z \in X_2$ be a non-neighbour of $u$.
                
		Owing to the missing edge $uz$, and by the edge-maximality of $G$,
		there must exist a $4$-path $uw_1w_2w_3z$ in $G$ between $u$ and $z$ (a
		$2$-path is impossible). Consider the $(5k-3)$-cycle $C =
		(uw_1w_2w_3zx_1x_{5k-8}\ldots x_3y)$, where $x_i \in X_i$ (see
		\Cref{fig:twelve-to-seventeen-cycle}). 

        \begin{figure}[h]
            \centering
            \includegraphics[scale = .85]{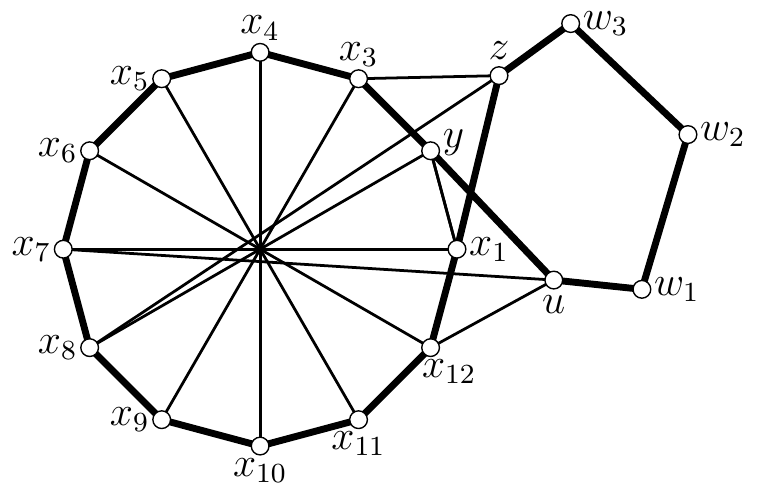}
            \caption{the $(5k - 3)$-cycle $C$ obtained from $u$ and $H$, $k = 4$}
            \label{fig:twelve-to-seventeen-cycle}
        \end{figure}

		Our aim is to show that $V(C)$ induces a copy of $F_k$, contrary to our
		assumption on $G$. Relabel the cycle $C$ in order as $(z_0z_1\ldots
		z_{5k-4})$, so that $z_0 = u, z_i = w_i$ for $i=1, 2, 3$, $z_4 = z,
		z_5= x_1$, $z_i = x_{5k-2-i}$ for $6\leq i \leq 5k-5$, and $z_{5k-4} =
		y$.  We must check that all chords of lengths $1+5t$ for $t=0, \ldots
		k-1$ are present in the graph induced on $V(C)$. Note that all possible
		chords of these lengths that are not incident with a vertex in $S=\{u,
		w_1, w_2, w_3\} = \{z_0, z_1, z_2, z_3\}$ are present, since all
		vertices in $V(C)\setminus S$ are in an appropriate copy of $F_{k-1}$.
		So we must check that all possible chords incident with a vertex in $S$
		are present. This is summarized in the following claim, where we
		temporarily revert to the original labelling of $C$:
        
        \begin{claim}\label{claim:nbhds-of-y-i}
        	The following hold:
            \begin{itemize}
                \item 
                    $N(u, C) = \{w_1, y\}\cup \{x_{5\ell+2}: 1\leq \ell \leq
                    k-2\}$.
                \item 
                    $N(w_1, C) = \{u, w_2\}\cup \{x_{5\ell +1}: 1\leq \ell
                    \leq k-2\}$.
                \item 
                    $N(w_2, C) = \{w_1, w_3\}\cup \{x_{5\ell}: 1\leq \ell
                    \leq k-2\}$.
                \item 
                    $N(w_3, C) = \{z, w_2\}\cup \{x_{5\ell-1}: 1\leq \ell
                    \leq k-2\}$.
            \end{itemize}
        \end{claim}	
        
        \begin{proof}
			Observe that the first item is immediate from our choice of $u$.
			Fix some $\ell$ with $1\leq \ell \leq k-2$. Note that every vertex
			in $X_2$ is joined to $x_{5(\ell-1)+3}=x_{5\ell-2}$. In particular,
			$y$ and $z$ are joined to $x_{5\ell -2}$. Consider the $12$-cycle
			$C'= (uw_1w_2w_3zx_{5\ell-2}\ldots x_{5\ell+2}x_1y)$, with two
			consecutive diagonals $yx_{5\ell-2}$ and $x_1z$. Observe that $C'$
			gives rise to another $12$-cycle $C'' = (uw_1w_2w_3zx_1x_{5\ell
			+2}x_{5\ell+1}\ldots x_{5\ell -2}y)$ with two consecutive diagonals
			$yx_1$ and $ux_{5\ell+2}$.  By \Cref{lem:cycle-chords}, either $C'$
			or $C''$ has all of its diagonals present. However, it cannot be
			$C'$, since $u$ cannot be adjacent to $x_{5\ell-1}$.  Therefore,
			$C''$ has all diagonals present: $w_1x_{5\ell+1}$, $w_2x_{5\ell}$,
			and $w_3x_{5\ell-1}$ are edges in $G$.  This completes the proof of
			\Cref{claim:nbhds-of-y-i}.
        \end{proof}
                
		It remains to check that \Cref{claim:nbhds-of-y-i} produces chords of
		the right lengths. We do this for chords incident with $w_1$; the other
		cases follow identically. Indeed, $w_1 = z_1$ so we must check that
		$z_1$ is joined to $z_{1+(1+5t)}$ for $t=0, 1, \ldots , k-1$. This is
		obviously true for $t=0$ and $t=k-1$, so let $1\leq t \leq k-2$. Then
		the above is equivalent to $w_1$ being joined to $x_{5k-2- (1+(1+5t))}
		= x_{5(k-t-1)+1}$, where $1\leq k-t-1 \leq k-2$, which clearly follows
		by \Cref{claim:nbhds-of-y-i}. Accordingly, there is a copy of $F_k$ in
		$G$ contrary to our assumption, so $u$ must be adjacent to every vertex
		in $X_j$ for all $j \in J$. But then we may place $u$ in $X_1$ and
		produce a blow-up of $F_{k-1}$ of larger order, which is a
		contradiction to the choice of $H$.  It follows that every vertex
		in $V(G)\setminus V(H)$ is adjacent to vertices in at most $k-2$ of the
		sets $X_i$. In fact, by (\ref{itm:F-k-neighbd}) of \Cref{prop:Fk}, it
		follows that every vertex in $V(G) \setminus V(H)$ is adjacent to
		precisely $k - 2$ of the $X_i$'s.
        
		Before proceeding, let us introduce a bit of notation and terminology.
		Let $\widetilde{H}$ be the graph with vertex set $\{X_1, \ldots ,
		X_{5k-8}\}$, where an edge $X_iX_j$ is present whenever the pair
		$\{X_i, X_j\}$ induces a complete bipartite graph in $G$. As
		$H$ is a blow-up of $F_{k-1}$, $\widetilde{H}$ is isomorphic to
		$F_{k-1}$. We say that a vertex $v$ is \emph{joined} to a subset $X
		\subseteq V(G)$ if $v$ is adjacent to every vertex of $X$.
        
		If a vertex $v$ is joined to vertices in the neighbourhood of $X_i$,
		then by (\ref{itm:F-k-neighbd}) of \Cref{prop:Fk} we have that $v$
		misses vertices in at most the two sets $X_{i-1}, X_{i+1}$; by
		symmetry, we may assume that each such vertex $v$ misses $X_{i-1}$.
		Thus the following sets $Y_i$, where $i = 1, \ldots, 5k - 8$, defined
		below, form a partition of $V(G) \setminus V(H)$ (see
		\Cref{fig:twelve-cycle-Yi}).  Note that each of these sets is
		independent (as $G$ is triangle-free):
        \begin{equation*}
            Y_i = \{u \in V(G) \setminus V(H): \text{$u$ is joined to $X_{i +
            1}, X_{i + 6}, \ldots, X_{i + 5k - 14}$} \text{ (indices modulo $5k-8$)}\}
        \end{equation*}

        \begin{figure}[h]
            \centering
            \includegraphics[scale = .85]{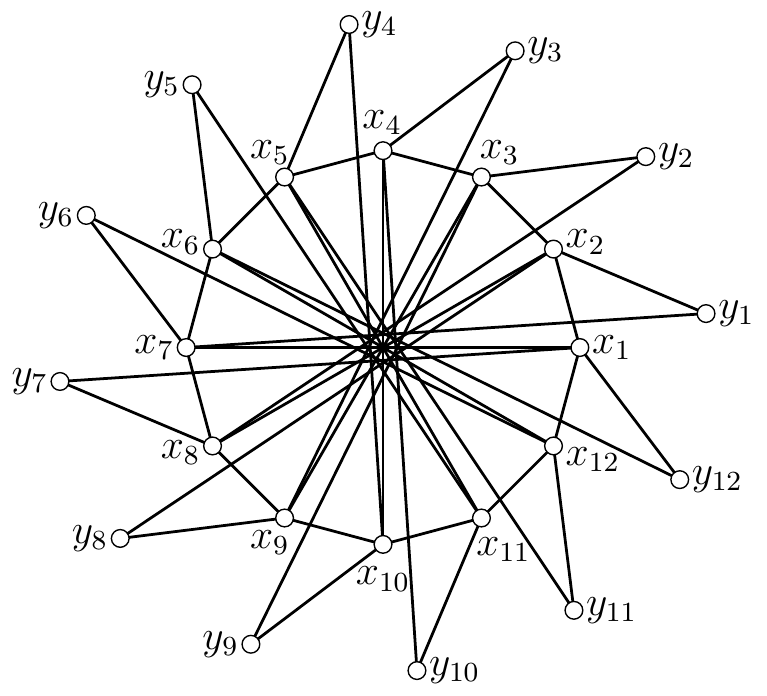}
            \caption{the sets $X_i$ and $Y_i$}
            \label{fig:twelve-cycle-Yi}
        \end{figure}
         
        \begin{claim}\label{claim:Yi-Yj-edges}
			Let $k\geq 3$ and $1\leq i, j \leq 5k-8$. If $j$ is such that $X_j
			\notin N_{\widetilde{H}}(X_i)$, then there are no edges between
			$Y_i$ and $Y_j$.
        \end{claim}
                
        \begin{proof}
			Without loss of generality, set $i=1$. Suppose $j$ is such that
			$X_j \notin N_{\widetilde{H}}(X_1)$.  We may assume that $j\neq 1$,
			as each $Y_i$ is independent. Then $j = 5l + r$, where $l \in
			\{0,\ldots, k - 3\}$ and $r \in \{3,4,5,6\}$.  Towards a
			contradiction, suppose there is an edge $y_1y_j$ between $Y_1$ and
			$Y_j$.  We consider four cases, according to the value of $r$.
			Suppose first that $r = 3$.  Then we find the following $5$-cycle
			$(y_1 x_2 x_{5l + 3} x_{5l + 4} y_j)$.  If $r = 4$, we find the
			induced $6$-cycle $(y_1 x_2 x_{5l + 3} x_{5l + 4} x_{5l + 5} y_j)$.
			If $r = 5$, there is, again, an induced $6$-cycle $(y_1 x_2 x_1
			x_{5l+7} x_{5l+6} y_j)$.  Finally, if $r = 6$, there is a $5$-cycle
			$(y_1 x_2 x_{5l + 8} x_{5l + 7} y_j)$.

			For each of the possible values of $r$, we reached a contradiction
			by showing that $G$ contains either a $5$-cycle or an induced
			$6$-cycle. \Cref{claim:Yi-Yj-edges} follows.
        \end{proof}

		Let $Z_i = X_i \cup Y_i$. Note that the sets $Z_i$ are independent and
		they partition $V(G)$.  It follows from \Cref{claim:Yi-Yj-edges} that
		there are no $Z_i - Z_j$ edges if $X_i X_j \notin E(\widetilde{H})$. By
		maximality of $G$, all $Z_i - Z_j$ edges are present if $X_i X_j \in
		E(\widetilde{H})$, implying that $G$ is a blow-up of $F_{k - 1}$. In
		particular, $G$ is homomorphic to $F_{k - 1}$, as required to complete
		the proof of \Cref{thm:hom-or-Fk}.
    \end{proof}
        
	We close this section by showing the following consequence of
	\Cref{thm:hom-or-Fk}.
    
	\begin{cor} \label{cor:hom-F-k-min-deg}
        Let $G$ be a \cfree{} graph on $n$ vertices with $\delta(G) >
        \frac{k}{5k - 3} \, n$. Then $G$ is homomorphic to $F_{k - 1}$.
    \end{cor}
    	
   	\begin{proof} 
		Note that we may assume that $G$ is maximal \cfree{}.  By
		\Cref{thm:hom-or-Fk}, if $G$ is not homomorphic to $F_{k - 1}$, it
		contains a copy $F$ of $F_k$.  The number of edges between $V(F)$ and
		$V(G) \setminus V(F)$ is at most $k(n - (5k - 3))$, since every vertex
		in $G$ has at most $k$ neighbours in $F$, by \Cref{prop:Fk}. It follows
		that there is a vertex $u$ in $F$ with at most $\frac{k n}{5k - 3} - k$
		neighbours outside of $F$.  Since $u$ has $k$ neighbours in $F$, it
		follows that $u$ has degree at most $\frac{kn}{5k - 3}$, a
		contradiction to the minimum degree condition.
    \end{proof}
  
\section{Homomorphism thresholds}\label{sec:hom-thresholds}

	Recall that, given a family of graphs $\MH$, the homomorphism threshold
	$\delta_{\hom}(\MH)$ of $\MH$ is the infimum of $d$ such that every
	$\MH$-free graph with $n$ vertices and minimum degree at least $dn$ is
	homomorphic to a bounded $\MH$-free graph.  In this section, we prove
	\Cref{thm:hom-threshold}, thereby determining the value of
	$\delta_{\hom}(\{C_3, C_5\})$. We also prove that $\delta_{\hom}(C_5) \le
	1/5$ by showing that $C_5$-free graphs of large enough minimum degree are
	also triangle-free.
   
    \begin{proof}[ of \Cref{thm:hom-threshold}]
		Denote $\delta = \delta_{\hom}(\{C_3, C_5\})$.  First, we show that
		$\delta \ge 1/5$.  We note that $F_k$ is not homomorphic to a \cfree{}
		graph $H$ with fewer than $|V(F_k)|$ vertices. Indeed, suppose
		otherwise. Then two vertices $x$ and $y$ in $F_k$ are mapped to the
		same vertex $u$ in $H$. By (\ref{itm:F-k-paths-between-vs}) there is a
		path $P$ of length $1$, $3$ or $5$ between $x$ and $y$. Clearly, $P$
		cannot have length $1$ (because the set of vertices mapped to the same
		vertex is independent).  It follows that $P$ has length $3$ or $5$.
		This implies that the path $P$ is mapped to a cycle of length $3$ or
		$5$, a contradiction.  It follows that, for each $k \ge 1$, $F_k$ is a
		\cfree{} graph with minimum degree at least $|V(F_k)| / 5$, which is
		not homomorphic to a \cfree{} graph on fewer than $|V(F_k)|$ vertices.
		Hence, indeed, $\delta \ge 1/5$. 

		It remains to show that $\delta \le 1/5$. Let $\eps > 0$ be fixed.
		Suppose that $G$ is a \cfree{} on $n$ vertices and minimum degree at
		least $(1/5 + \eps)n$. Let $k$ be such that $\frac{k}{5k - 3} < 1/5 +
		\eps$. Then, by \Cref{cor:hom-F-k-min-deg}, $G$ is homomorphic to $F_{k
		- 1}$. This shows that $\delta \le 1/5 + \eps$. Since $\eps$ was
		arbitrary, we conclude that $\delta \le 1/5$. 
    \end{proof}

	It would be interesting to determine the homomorphism threshold of
	$C_5$. The following lemma enables us to easily obtain an upper bound.

	\begin{lem} \label{lem:no-triangles}
		Let $G$ be a $C_5$-free graph on $n$ vertices and let $\gamma > 0$ with
		$\delta(G) \ge n/6 + \gamma n$. Then $G$ is triangle-free provided $n$
		is sufficiently large.
	\end{lem}
	
	We remark that with a little extra work we are able to prove the same
	result under the weaker condition that $\delta(G) > n/6 +1$, which is in
	fact tight whenever $12$ divides $n$; we omit the details for brevity.
	Before proving \Cref{lem:no-triangles}, we use it to prove
	\Cref{cor:hom-threshold-C-5}, which provides an upper bound on the
	homomorphism threshold $\delta_{\hom}(C_5)$. We currently do not have any
	nontrivial lower bound on $\delta_{\hom}(C_5)$.

	\begin{proof}[ of \Cref{cor:hom-threshold-C-5}]
		Suppose that $G$ is a $C_5$-free graph on $n$ vertices and minimum
		degree at least $(1/5 + \eps)n$ for some fixed $\eps > 0$.  Then, by
		\Cref{lem:no-triangles}, $G$ is also triangle-free.  It follows from
		\Cref{thm:hom-threshold} that $G$ is homomorphic to a $C_5$-free (and
		$C_3$-free) graph $H$ of order at most $C = C(\eps)$.  Hence, indeed,
		$\delta_{\hom}(C_5) \le 1/5$.  
	\end{proof}
		
	We now turn to the proof of \Cref{lem:no-triangles}.

	\begin{proof}[ of \Cref{lem:no-triangles}]
        
		We start by showing that every vertex in $G$ is incident with at most
		$13$ \emph{triangular edges} (i.e.~edges on triangles).  To see this,
		suppose that $u$ is incident with at least $14$ triangular edges.  In
		other words, the neighbourhood $N(u)$ of $u$ contains edges that span
		at least $14$ vertices. The following claim implies that there is a set
		$X$ of seven neighbours of $u$ such that every vertex in $X$ has a
		neighbour in $N(u) \setminus X$.

		\begin{claim} \label{claim:few-common-neighs}
			Let $H$ be a graph with $n$ vertices and no isolated vertices. Then
			there is a set $X$ of size at least $n / 2$ such that every vertex
			in $X$ has a neighbour outside of $X$.
		\end{claim}

		\begin{proof}
			We note that it suffices to prove the claim under the assumption
			that $H$ is connected. Indeed, for each component $H_i$ of $H$, we
			may pick a set $X_i$ as in the claim, and let $X$ be the union of
			the $X_i$'s. So now we assume that $H$ is connected. Because of
			the assumption that there are no isolated vertices, we may assume
			that $|V(H)| \ge 2$.

			Let $u$ be a vertex for which $H \setminus \{u\}$ is connected.
			Let $v$ be a neighbour of $u$. Consider the graph $H' = H \setminus
			\{u, v\}$. Let $H_1, \ldots, H_t$ be the connected components of
			$H'$. We pick a set $X_i$ for each $i \in [t]$ as follows: if $H_i$
			consists of a single vertex $x_i$, then $x_i$ must be adjacent to
			$v$, and we take $X_i = \{x_i\}$; otherwise, if $H_i$ has at least
			two vertices, then by induction there is a set $X_i$ of size at
			least $|V(H_i)|/2$ such that every vertex in $X_i$ has a neighbour
			outside of $X_i$ (but in $H_i$). Let $X=\bigcup_{i=1}^t X_i \cup
			\{u\}$. It is easy to check that $X$ satisfies the requirements of
			\Cref{claim:few-common-neighs}.
		\end{proof}
            
		Let $Y$ be a set of at most seven neighbours of $u$, which is disjoint
		from $X$ and satisfies that every vertex in $X$ has a neighbour in $Y$.
		Due to the minimum degree condition, if $n$ is sufficiently large, then
		we may find two distinct vertices $x_1$ and $x_2$ in $X$ that have a
		common neighbour $z$ outside of $X \cup Y \cup \{u\}$.  Let $y \in Y$
		be a neighbour of $x_1$. Then we find the $5$-cycle $(x_1 y u x_2 z)$,
		a contradiction. Thus, indeed, every vertex is incident with at most
		$13$ triangular edges.
		
		With this in mind, the proof of \Cref{lem:no-triangles} is nearly
		complete.  Indeed, suppose that $T = x_1x_2x_3$ is a triangle in $G$,
		and remove an edge from each triangle (except $T$) to form a new graph
		$G'$. The minimum degree only drops by at most $13$, so if $n$ is
		sufficiently large, we obtain $\delta(G') > n/6 + 1$.  Consider the
		neighbourhoods $N_1, N_2, N_3$ of $x_1, x_2, x_3$ outside of $T$. These
		sets are pairwise disjoint and independent in $G'$. Pick $y_i \in N_i$
		and consider $N(y_i)\setminus \{x_i\}$ for $i = 1, 2, 3$. These sets
		are pairwise disjoint (from the assumption that $G$ contains no
		$5$-cycle), and independent. Moreover, they are disjoint from $N_1 \cup
		N_2 \cup N_3$.  Letting $M$ be the union of these six sets and $T$, we
		have
        \[
           	|M| > 3\left(\frac{n}{6} -1\right) + 3\cdot\frac{n}{6} + 3 = n,
		\]
		a contradiction. Thus $G$ is triangle-free provided $n$ is sufficiently large, completing the proof.
	\end{proof}

\section{Final remarks}

	We are able to determine precisely the structure of \cfree{} graphs with
	high minimum degree, and thereby deduce the value of the homomorphism
	threshold $\delta_{\hom}(\{C_3, C_5\})$.  It would be very interesting to
	extend this result to $\{C_3, \ldots , C_{2\ell - 1}\}$-free graphs.
	Recall that, for integers $k\ge 2, \ell \ge 3$, $F_k^\ell$ is the graph
	obtained from a $((2\ell-1)(k-1)+2)$-cycle by adding all chords joining
	vertices at distances $j(2\ell-1)+1$ for $j=0, 1, \ldots , k-1$. In light
	of our \Cref{thm:hom-to-family} it is natural to ask whether or not a
	$\{C_3, C_5 \ldots , C_{2\ell -1}\}$-free graph on $n$ vertices  with
	minimum degree larger than $\frac{n}{2\ell-1}$ is homomorphic to $F_k^\ell$
	for some $k$. Rather surprisingly it turns out that this is false when
	$\ell \geq 4$ is even, as shown by the following construction due to Oliver
	Ebsen \cite{ebsen-schacht}. Suppose that $\ell \ge 4$ is even. Starting
	with a complete graph on four vertices, subdivide two independent edges by
	an additional $2\ell-6$ vertices and subdivide the remaining four edges by
	an additional two vertices each. Denote the resulting graph by $T_\ell$. It
	is easy to check that this graph is maximal $\{C_3, C_5 \ldots ,
	C_{2\ell-1}\}$-free.  To obtain large minimum degree assign weight $2$
	to each vertex of the original $K_4$ and to $\ell-4$ vertices of the
	`long' subdivided edges, and assign weight $1$ to the remaining
	vertices. This may be done in such a way that each vertex has weight
	$3$ in its neighbourhood (as $\ell$ is even). To obtain an unweighted
	graph of order $n$ simply blow up each vertex with an independent set
	of size proportional to its weight. Then the resulting graph
	$T^*_{\ell}$ is maximal $\{C_3, C_5 \ldots , C_{2\ell-1}\}$-free and
	$\delta(T^*_\ell) = \frac{3n}{6\ell-4} > \frac{n}{2\ell-1}$. However,
	it is not hard to show that $T_\ell$ is not homomorphic to $F_k^\ell$,
	for any $k$ (and therefore no blow-up of $T_\ell$ is homomorphic to any
	$F_k^\ell$).  We do not know whether \Cref{thm:hom-to-family} extends
	naturally to $\{C_3, C_5, \ldots, C_{2\ell-1}\}$-free graphs when $\ell
	\ge 5$ is odd, and it would be interesting to pursue this line of
	research further.
    
	Recall that the homomorphism threshold of a family of graphs $\MH$ is the
	infimum of $d$ satisfying that every $\MH$-free graph with $n$ vertices and
	minimum degree at least $dn$ is homomorphic to an $\MH$-free graph of
	bounded order (depending on $d$ but not on $n$). Despite the above remarks
	concerning the extension of \Cref{thm:hom-to-family} to general odd-girth
	graphs, we still make the following conjecture concerning the homomorphism
	threshold of $\{C_3, C_5, \ldots , C_{2\ell-1}\}$-free graphs for $\ell \ge 4$.
    
    \begin{conj} \label{conj:general-odd-girth}
    
		Let $\ell \ge 4$ be an integer. Then $\delta_{\hom}(\{C_3, C_5 \ldots ,
		C_{2\ell-1}\}) = \frac{1}{2\ell-1}$.
            
    \end{conj}
    
    We have also obtained an upper bound on $\delta_{\hom}(C_5)$, namely, that
    it is at most $1/5$. We ask if it is true that $1/5$ is the correct value.
	\begin{qn} \label{qn:delta-5}
        Is it true that $\delta_{\hom}(C_5) = 1/5$?
    \end{qn}
    
	In fact, any nonzero lower bound on $\delta_{\hom}(C_5)$ would be
	interesting. In order to obtain such a lower bound, one would have to find,
	in particular, a family of graphs that have large minimum degree, are
	$C_5$-free and are not $4$-colourable (indeed, otherwise, the graphs are
	homomorphic to $K_4$, which is clearly $C_5$-free). Although it is well
	known that such graphs exist, it seems hard to find explicit examples,
	especially with the added condition that they are not homomorphic to
	$C_5$-free graphs of bounded order.

	\paragraph{Note added in the proof.} 
		After the paper was submitted Ebsen and Schacht
		\cite{ebsen-schacht-new} proved that the homomorphism threshold of
		$\{C_3, \ldots, C_{2\ell-1}\}$ is $\frac{1}{2\ell-1}$, thereby proving
		\Cref{conj:general-odd-girth}. We note that \Cref{qn:delta-5} is still
		left unanswered. Furthermore, in light of this new result, we pose the
		following question which suggests a strengthening of
		\Cref{conj:general-odd-girth}.
		\begin{qn}
			Let $\ell \ge 4$ be integer and let $\eps > 0$, and suppose $G$ is a
			$\{C_3, \ldots, C_{2\ell-1}\}$-free graph on $n$ vertices with
			minimum degree at least $(\frac{1}{2\ell - 1} + \eps)n$. Is it true
			that $G$ is $3$-colourable?
		\end{qn}
	 
	\subsection*{Acknowledgments}

		This research was initially conducted at the University of Cambridge. The
		second author is grateful to the Department of Mathematical Sciences at the
		University of Memphis, and especially to B{\'e}la Bollob{\'a}s, for making
		that possible, and for the hospitality of the Combinatorics Group at
		Cambridge.  We should also like to thank Julian Sahasrabudhe for
		suggesting the problem, and B{\'e}la Bollob{\'a}s, Tomasz {\L}uczak,
		and Mathias Schacht for very helpful comments. Finally, we would 
		like to thank the anonymous referees for their careful reading and helpful
		suggestions.

    \bibliography{hombib}
    \bibliographystyle{amsplain}
\end{document}